%% file: ProjectedGD.tex
\documentclass[11pt]{article}
\usepackage{algorithm,algorithmicx,algpseudocode,amsmath,amssymb,amsthm}
\usepackage{bm,color,dsfont,epsfig,fullpage,graphicx,pifont,subfigure}
\usepackage{makecell}
\usepackage[toc,page]{appendix}
\usepackage{multirow}
\usepackage{scalerel}
\usepackage{url}

\usepackage{authblk}
\newcommand{\myfootnote}[1]{
\renewcommand{\thefootnote}{}
\footnotetext{\hspace{-16.5pt}\scriptsize#1}
\renewcommand{\thefootnote}{\arabic{footnote}}}

\newtheorem{definition}{Definition}[section]

\newtheorem{lemma}{Lemma}[section]
\newtheorem{theorem}{Theorem}[section]
\newtheorem*{remark*}{Remark}
\newtheorem{proposition}{Proposition}[section]
\newtheorem*{example*}{Example}

\allowdisplaybreaks
\newcommand\numberthis{\addtocounter{equation}{1}\tag{\theequation}}

\def\la{\left\langle}
\def\ra{\right\rangle}
\def\lb{\left(}
\def\rb{\right)}
\def\lcb{\left\{}
\def\rcb{\right\}}
\def\ln{\left\|}
\def\rn{\right\|}
\def\lsb{\left[}
\def\rsb{\right]}
\def\lab{\left|}
\def\rab{\right|}
\def\C{\mathbb{C}}

\def\P{\mathcal{P}}
\def\T{\mathcal{T}}
\def\H{\mathcal{H}}
\def\D{\mathcal{D}}
\def\G{\mathcal{G}}
\def\R{\mathbb{R}}
\def\E{\mathbb{E}}

\def\CS{\mathcal{C}}
\def\I{\mathcal{I}}
\def\BZ{\bm{Z}}
\def\TBZ{{\bm{\tilde{Z}}}}

\def\BL{\bm{L}}
\def\BU{\bm{U}}
\def\BV{\bm{V}}
\def\BB{\bm{B}}
\def\BC{\bm{C}}
\def\BD{\bm{D}}
\def\BE{\bm{E}}
\def\BM{\bm{M}}
\def\BQ{\bm{Q}}
\def\BH{\bm{H}}
\def\BX{\bm{X}}
\def\BA{\bm{A}}
\def\BB{\bm{B}}
\def\BC{\bm{C}}
\def\BD{\bm{D}}
\def\BG{\bm{G}}
\def\by{\bm{X}}
\def\BY{\bm{M}}
\def\BYY{\bm{Y}}
\def\BR{\bm{Q}}
\def\BHH{\bm{H}}
\def\BS{\bm{\Sigma}}
\def\BI{\bm{I}}
\def\BLam{\bm{\Lambda}}

\def\SU{{\scaleto{\bm{U}\mathstrut}{6pt}}}
\def\SV{{\scaleto{\bm{V}\mathstrut}{6pt}}}

\def\SZ{{\scaleto{\bm{Z}\mathstrut}{6pt}}}
\def\TSZ{{\scaleto{\bm{\tilde{Z}}\mathstrut}{6pt}}}
\def\ba{\bm{a}}
\def\bx{\bm{x}}
\def\be{\bm{e}}
\def\bz{\bm{z}}
\def\by{\bm{y}}
\def\bb{\bm{b}}
\def\bc{\bm{c}}
\def\bd{\bm{d}}
\def\bw{\bm{w}}

\def\bone{\bm{1}}

\def\bzero{\bm{0}}
\DeclareMathOperator{\rank}{rank}
\DeclareMathOperator{\diag}{diag}
\DeclareMathOperator{\Real}{{\normalfont Re}}
\DeclareMathOperator{\trace}{trace}
\newcommand{\PC}[1]{\P_{\CS}(#1)}
\newcommand{\mean}[1]{\E\lsb#1\rsb}
\newcommand{\prob}[1]{\mathbb{P}\lcb#1\rcb}
\newcommand{\dist}[2]{{\normalfont\mbox{dist}}(#1,#2)}
\newcommand{\distsq}[2]{{\normalfont\mbox{dist}^2}(#1,#2)}
\newcommand{\kw}[1]{#1}
\makeatletter
\renewcommand*{\@fnsymbol}[1]{\ensuremath{\ifcase#1\or \dagger\or \ddagger\or \mathsection\or
    *\or \mathparagraph\or \|\or **\or \dagger\dagger
    \or \ddagger\ddagger \else\@ctrerr\fi}}
\makeatother
\begin{document}

\title{Spectral Compressed Sensing via Projected Gradient Descent}
\author[1]{Jian-Feng Cai}
\author[1, 2]{Tianming Wang}
\author[3]{Ke Wei}
\affil[1]{Department of Mathematics, Hong Kong University of Science and Technology, Clear Water Bay, Kowloon, Hong Kong SAR, China.\vspace{.15cm}}
\affil[2]{Department of Mathematics, University of Iowa, Iowa City, Iowa, USA.\vspace{.15cm}}
\affil[3]{Department of Mathematics, University of California at Davis, Davis, California, USA.}
\myfootnote{\indent\indent Email addresses: jfcai@ust.hk (J.-F. Cai), tianming-wang@uiowa.edu (T. Wang), and weike1986@gmail.com (K. Wei, corresponding author).}

\maketitle

\begin{abstract}
Let $\bx\in\C^n$ be a spectrally sparse signal consisting of $r$  complex sinusoids with or without damping. We consider  the spectral compressed sensing problem, which is about  reconstructing $\bx$ from its partial revealed entries. By utilizing the low rank structure of the Hankel matrix corresponding to $\bx$, we develop a computationally efficient algorithm for this problem. The algorithm starts from an initial guess computed via one-step hard thresholding followed by projection, and then proceeds by applying projected gradient descent iterations to a non-convex functional. Based on the sampling with replacement model, we prove that $O(r^2\log(n))$  observed entries are sufficient for our algorithm to achieve the successful recovery of a spectrally sparse signal. Moreover, extensive empirical performance comparisons show that our  algorithm is competitive  with other state-of-the-art spectral compressed sensing algorithms in terms of phase transitions and overall computational time.

\vspace{.5cm}
\end{abstract}

\input{introduction}
\input{algorithm}
\input{numerics}

\input{proofs}
\input{discussion}
\input{appendix}

\bibliographystyle{siam}
\bibliography{references}



\end{document}

%% file: introduction.tex
\section{Introduction}\label{sec:intro}
\subsection{Problem Setup}
In this paper, we are interested in the problem of reconstructing  a spectrally sparse signal  with or without damping from its nonuniform time-domain samples. Let $x(t)$ be a one-dimensional signal. We say that $x(t)$ is spectrally sparse if it is superposition of a few complex sinusoids, namely 
\begin{align*}
x(t) = \sum_{k=1}^rd_ke^{(2\pi \imath f_k-\tau_k)t},\numberthis\label{eq:signal_model}
\end{align*}
where $\imath=\sqrt{-1}$, $r$ is the model order, $f_k$ is the frequency of each sinusoid, $d_k$ is the weight of each sinusoid, and $\tau_k\geq 0$ is a damping factor. Let $n>0$ be a natural number. Without loss of generality, we assume $f_k\in[0,1)$ and consider the samples of $x(t)$ at all the integer values from $0$ to $n-1$, denoted  $\bx$. 
That is,
\begin{align*}
\bx = \begin{bmatrix}
x(0)&\cdots&x(n-1)
\end{bmatrix}^T\in\C^n.\numberthis\label{eq:sampling_model}
\end{align*}

Spectrally sparse signals of the form \eqref{eq:signal_model} and the corresponding sampling model in \eqref{eq:sampling_model} arise in many areas of science and engineering including magnetic resonance imaging \cite{MRI}, fluorescence microscopy \cite{Microscopy}, radar imaging \cite{Radar},  nuclear magnetic resonance spectroscopy \cite{QMCCO:ACIE:15}, and analog-to-digital conversion \cite{Analog}. However, in  those real-world applications, full sampling at all the points on a uniform grid is either time-consuming or technically prohibited. In addition, the signal may become too weak to be detected after a certain period of time when $\tau_k>0$. Therefore, for the purpose of more efficient data acquisition, nonuniform sampling is typically used in practice. When restricted to the sampling model in \eqref{eq:sampling_model}, this means that only partial entries of $\bx$ are known and we need to  estimate the missing ones. Let $\Omega$ be subset of $\{0,\cdots,n-1\}$ corresponding to the observed entries, and let $\P_\Omega$ be the associated sampling operator which acquires only the entries indexed by $\Omega$. Then the task can be formally expressed as:
\begin{align*}
\mbox{Find}\quad\bx\quad\mbox{subject to}\quad\P_\Omega(\bx)=\sum_{a\in\Omega}x_a\be_a,\numberthis\label{eq:setup}
\end{align*}
where $\{\be_a\}_{a=0}^{n-1}$ is a canonical basis of $\C^n$. 
In the sequel, we shall refer to the vector $\bx$ as {\em a spectrally sparse signal}, and refer to the problem of reconstructing a spectrally sparse signal from its partial observed entries as {\em spectral compressed sensing} or {\em spectrally sparse signal recovery}.
\subsection{Prior Art and Main Contributions}
It is clear that \eqref{eq:setup} is a task that cannot be achieved if $\bx$ does not have  any intrinsic simple structures. 
Fortunately, the signal of interest in this paper is spectrally sparse. Moreover,
the number of degrees of freedom in $\bx$ is completely determined by the number of Fourier modes in $x(t)$, which is proportional to $r$ and independent of $n$. This key observation suggests the possibility of reconstructing $\bx$ from its partial revealed entries, which can be further achieved by exploiting  the simplicity
of $\bx$ in different ways.

Note that we are mainly interested in the scenario where $\bx$ only has a few Fourier components (i.e., $r$ is small). Thus, one can utilize the sparsity of $\bx$ in the frequency domain to design reconstruction algorithms. 
In particular, if there is no damping in $\bx$, spectral compressed sensing can be recast as a conventional compressed sensing problem \cite{donoho2006cs, CS} after discretization of the Fourier domain; so many existing algorithms  for compressed sensing are available, such as  Basis Pursuit \cite{CDS98},  
IHT \cite{bludav2009iht, blumensathdavies2010niht,CGIHT, CGIHTnoise, foucart2011htp},  CoSaMP \cite{NeTr_cosamp} and SP \cite{SubspacePursuit}. 
However, the performance of the compressed sensing approach for spectrally sparse signal recovery suffers from the mismatch error between the true frequencies and the discrete frequencies \cite{Mismatch,Strohmer_myad}. A grid-free approach was developed in \cite{Tang} which exploited the frequency sparsity of $\bx$ in a continuous manner via the atomic norm minimization (ANM). It was shown in \cite{Tang} that ANM could achieve exact recovery from $O(r\log(r)\log(n))$ random time-domain samples under some mild conditions.

By the Vandermonde decomposition, one may easily see that the Hankel matrix computed from  a spectrally sparse signal is low rank when $r$ is small relative to $n$.
 Consequently, spectral compressed sensing can be reformulated as a low rank Hankel matrix completion problem\footnote{See Section~\ref{sec:alg:hankel} for details.}. 
Inspired by low rank matrix completion \cite{candesrecht2009mc}, another grid-fee method known as enhanced matrix completion (EMaC) was developed in \cite{Chi}  by reformulating the non-convex low rank Hankel matrix completion problem into a convex Hankel matrix nuclear norm minimization problem. EMaC was shown to be able to reconstruct a spectrally sparse signal with high probability provided the number of observed entries is $O(r\log^4(n))$. The same approach was studied in \cite{CQXY:ACHA:16} under the Gaussian random sampling model, and various first-order methods were discussed in \cite{Hankel_Fazel} for the regularized Hankel matrix nuclear norm minimization problem.
Alternative to EMaC, there have been several non-convex algorithms which were designed to  solve the low rank Hankel matrix completion directly. Examples include PWGD \cite{PWGD}, IHT and FIHT \cite{FIHT}. Compared to the convex approaches such as ANM and EMaC, those non-convex algorithms  are typically much more efficient, especially for higher dimensional problems. Moreover, inspired by the guarantee analysis of Riemannian optimization for low rank matrix reconstruction \cite{Recovery, Completion}, it was shown in \cite{FIHT} that FIHT with a proper initial guess was able to reconstruct a spectrally sparse signal with high probability from $O(r^2\log^2(n))$ random observations. For multi-dimensional spectrally sparse signal recovery problems, we can also exploit the low rank tensor structure of the signal when developing recovery algorithms, see for example \cite{Hankel_Tensor} and references therein.

{\em The main contributions of this work are two-fold.} Firstly,  we present a new non-convex algorithm for spectral compressed sensing via low rank Hankel matrix completion, which we refer to as Projected Gradient Descent (PGD).  Extensive empirical performance comparisons show that PGD is competitive with other state-of-the-art spectral compressed sensing algorithms both in terms of the problem size that can be solved and in terms of overall computation time. Secondly, exact recovery guarantee has been established for PGD, showing that PGD can successfully recover a spectrally sparse signal from $O(r^2\log(n))$ random observed entries.

Although we focus on spectrally sparse signal recovery in this paper, the proposed PGD algorithm can be  easily extended to the general low rank Hankel matrix completion problem. Moreover, the recovery guarantee analysis equally applies provided the underlying target matrix is incoherent\footnote{See Definition~\ref{def:coh}.}. Low-rank Toeplitz matrices can also be provably recovered from  partial revealed entries by a slightly modified version of PGD.
\subsection{Outline and Notation}
The remainder of this paper is organized as follows. We present the details of PGD, along with its recovery guarantee in Section~\ref{sec:alg}. In Section~\ref{sec:numerics} we evaluate the empirical performance of PGD with a set of numerical experiments.  The proof of the exact recovery guarantee  is presented in Section~\ref{sec:pf}. We conclude the paper with some potential future directions in Section~\ref{sec:discuss}.

Throughout the paper we use the following notational conventions.  We denote vectors by bold lowercase letters and matrices by bold uppercase letters, and the numbering of vector and matrix elements starts at zero. In particular, we fix $\bx$, $\by$, and $\BY$ as the target signal and its transformations. The individual entries of vectors and matrices are denoted in normal font. We denote by $\ln\BZ\rn_*$, $\ln\BZ\rn_2$ and $\ln\BZ\rn_F$ the nuclear norm, spectral norm and Frobenius norm of the matrix $\BZ$, respectively. Additionally, we define $\ln\BZ\rn_{2,\infty}$ as the largest $\ell_2$-norm of its rows. For a vector $\bz$, we use $\ln\bz\rn_1$ and $\ln\bz\rn_2$ to denote its $\ell_1$-norm and $\ell_2$-norm, respectively.  For
both vectors and matrices, $\bz^T$
and $\BZ^T$ denote their transpose while $\bz^*$ and $\BZ^*$ denote their conjugate transpose. The inner product of two matrices $\BZ_1$ and $\BZ_2$ is defined as $\la\BZ_1,\BZ_2\ra=\trace\lb\BZ_1^*\BZ_2\rb$. When restricted to two vectors $\bz_1$ and $\bz_2$, the inner product is given by $\la\bz_1,\bz_2\ra=\bz_1^*\bz_2$.  For a natural number $n$,  $[n]$ denotes the set $\{0,\cdots,n-1\}$.


Operators are denoted by calligraphic letters. In particular,  $\I$ denotes the identity operator and $\H$ denotes the linear  operator which maps $n$-dimensional vectors to  $n_1\times n_2$ Hankel matrices with $n_1+n_2=n+1$, i.e., for any vector $\bz\in\C^{n}$, $[\H\bz]^{(i,j)} = z_{i+j}$ for $i\in[n_1]\mbox{ and } j\in[n_2]$.  The ratio $c_s$ is defined as $c_s=\max\{n/n_1,n/n_2\}$.
We denote  the adjoint of $\H$ by $\H^*$, which is a linear operator from $n_1\times n_2$ matrices to $n$-dimensional vectors. For any matrix $\BZ\in\C^{n_1\times n_2}$, a simple calculation yields that $\H^*\BZ =\lcb\sum_{i+j=a}Z^{(i,j)}\rcb_{a=0}^{n-1} $.  Define $\D^2=\H^*\H$. It is easily verified that $\D$ is a linear operator from vectors to vectors which scales each entry of an $n$-dimensional vector by $\sqrt{w_a}$, where $w_a$  is the number of elements in the $a$-th skew-diagonal of an $n_1\times n_2$ matrix. Define $\G=\H\D^{-1}$ and let $\G^*$ be the adjoint of $\G$. One can easily see that the following orthogonal property holds: $\G^*\G=\I$. Finally, we use $c$, $c_1$, $c_2$, $\cdots$ to denote positive absolute numerical constants whose values may change from place to place.

%% file: algorithm.tex
\section{Algorithm and Main Result}\label{sec:alg}
\subsection{Expoiting Low Rank Structure}\label{sec:alg:hankel}
As noted in the introduction, it is impossible to recover   a signal from its partial known entries if there are no hidden  simple structures. For a spectrally sparse signal, we can exploit its simplicity via the low rank structure of the corresponding Hankel matrix. Recall that a Hankel matrix is a matrix in which each skew-diagonal from left to right is constant. 
We define $\H$ as a linear operator which maps  a vector $\bz\in\C^n$ to an $n_1\times n_2$ ($n_1+n_2-1=n$) Hankel matrix, denoted $\H\bz$,  whose $i$-th skew-diagonal is equal to the $i$-th entry of $\bz$,
\begin{align*}
\H\bz =\begin{bmatrix}
z_0 & z_1 & z_2 & \cdots & \cdots& z_{n_2-1}\\
z_1 & z_2 &\cdots& \cdots & \cdots &  z_{n_2}\\
z_2 & \cdots& \cdots & \cdots & \cdots & z_{n_2+1}\\
\vdots & \vdots & \vdots & \vdots & \vdots &\vdots\\
z_{n_1-1} & z_{n_1} & \cdots & \cdots &  \cdots  &z_{n-1}
\end{bmatrix}.
\end{align*}
Thus, one has 
$
[\H\bz]^{(i,j)} = z_{i+j}$ for $ i\in[n_1]\mbox{ and } j\in[n_2].
$
In particular, the $(i,j)$-th entry of the Hankel matrix formed from the spectrally sparse signal $\bx$ is given by 
\begin{align*}
[\H\bx]^{(i,j)} = x_{i+j} = \sum_{k=1}^rd_ke^{(2\pi \imath f_k-\tau_k)(i+j)} = \sum_{k=1}^rd_ke^{i(2\pi \imath f_k-\tau_k)} e^{j(2\pi \imath f_k-\tau_k)}.
\end{align*}
If we let $w_k=e^{(2\pi\imath f_k-\tau_k)}$ for $k=1,\cdots,r$, it follows immediately that $\H\bx$ admits the following Vandermonde decomposition:
$$
\mathcal{H}\bm{x}=\bm{E}_L\bm{D}\bm{E}_R^T,
$$
where
$$
\bm{E}_L=
\left[\begin{array}{cccc}
1         & 1         & \cdots & 1 \\
w_1       & w_2       & \cdots & w_r \\
\vdots    & \vdots    & \vdots & \vdots \\
w_1^{n_1-1} & w_2^{n_1-1} & \cdots & w_r^{n_1-1} \\
\end{array}\right],~
\bm{E}_R=
\left[\begin{array}{cccc}
1         & 1         & \cdots & 1 \\
w_1       & w_2       & \cdots & w_r \\
\vdots    & \vdots    & \vdots & \vdots \\
w_1^{n_2-1} & w_2^{n_2-1} & \cdots & w_r^{n_2-1} \\
\end{array}\right]
$$
and
$\bm{D}=\diag(d_1,\cdots,d_r)$. Moreover, one has $\rank(\H\bx)=r$ provided the frequencies $\{f_k\}_{k=1}^r$ are different with each other and the diagonal entries of $\BD$ are all nonzeros.  

Obviously, each observed entry of $\bx$ corresponds to a revealed skew-diagonal of $\H\bx$. With a slight abuse of notation, denote by $\Omega$ the subset of the revealed skew-diagonals of $\H\bx$. Given a vector $\bz\in\C^n$, a simple calculation shows 
\begin{align*}
\la\P_\Omega(\H\bz-\H\bx),\H\bz-\H\bx\ra&=\sum_{a\in\Omega}\sum_{i+j=a}\lb[\H\bz]^{(i,j)}-[\H\bx]^{(i,j)}\rb^2\\
&=\sum_{a\in\Omega}w_a(z_a-x_a)^2\\
&=\la \P_\Omega\lb\D(\bz-\bx)\rb,\D(\bz-\bx)\ra,
\end{align*}
where $w_a$ in the second line is the number of entries in the $a$-th skew-diagonal of an $n_1\times n_2$ matrix, and $\D$ in the last line is a linear map which scales the $a$-th entry of a vector by a factor of $\sqrt{w_a}$ for all $a=0,\cdots,n-1$. 
We have seen that $\H\bx$ is a rank $r$ matrix. Thus, to reconstruct $\bx$, we may seek a signal $\bz$ such that $\rank(\H\bz)=r$ and $\H\bz$ fits the revealed skew-diagonals of $\H\bx$ as well as possible  by solving a rank constraint {\em weighted least square} problem:
\begin{align*}
\min_{\bz\in\C^n}\la \P_\Omega\lb\D(\bz-\bx)\rb,\D(\bz-\bx)\ra\quad\mbox{subject to}\quad\rank(\H\bz)=r.\numberthis\label{eq:low_rank_H}
\end{align*}

For ease of exposition, we will make a change of variables and rewrite \eqref{eq:low_rank_H} using  the new variable $\by=\D\bx$.   Denote by $\H^*$ the adjoint of $\H$, which maps a matrix $\BZ\in\C^{n_1\times n_2}$ to a vector $\H^*\BZ=\lcb\sum_{i+j=a}Z^{(i,j)}\rcb_{a=0}^{n-1}$. It is easy to show that $\H^*\H=\D^2$. Letting $\G=\H\D^{-1}$, we find that $\G$ has the desirable orthogonal property $\G^*\G=\I$, where $\I$ denotes the identity operator. After the substitution of $\D\bx$ by $\by$ and the substitution of $\D\bz$ by $\bz$, we can rewrite \eqref{eq:low_rank_H} as
\begin{align*}
\min_{\bz\in\C^n}\la \P_\Omega\lb\bz-\by\rb,\bz-\by\ra\quad\mbox{subject to}\quad\rank(\G\bz)=r,\numberthis\label{eq:low_rank_G}
\end{align*}
which will be our primary focus in this paper. A more direct interpretation of \eqref{eq:low_rank_G} is as follows.  Since $\by=\D\bx$, $\P_\Omega(\by) = \P_\Omega(\D\bx)=\D\P_\Omega(\bx)$,  $\rank(\G\by)=\rank(\H\bx)=r$, and $\D$ is invertible, one can instead attempt to reconstruct $\by$ from $\P_\Omega(\by)$  by seeking a signal that corresponds to  a low rank Hankel matrix and fits the observations as well as possible.
\subsection{Algorithm: Projected Gradient Descent}\label{sec:alg_pgd}
\subsubsection{Which Objective Function?} 
In order to eliminate the rank constraint in \eqref{eq:low_rank_G}, we parameterize $\G\bz$ by a product of two rank $r$ matrices and write $\G\bz$ as $\G\bz=\BZ_{\SU}\BZ_{\SV}^*$, where $\BZ_\SU\in\C^{n_1\times r}$ and $\BZ_\SV\in\C^{n_2\times r}$. \kw{We note that $\BZ_{\SU}\BZ_{\SV}^*$ is a Hankel matrix if and only if 
\begin{align*}
(\I-\G\G^*)(\BZ_{\SU}\BZ_{\SV}^*) = \bzero.
\end{align*}}Thus, by further noting that $\bz=\G^*(\G\bz) = \G^*(\BZ_{\SU}\BZ_{\SV}^*)$, we can rewrite \eqref{eq:low_rank_G}  using $\BZ_{\SU}$ and $\BZ_{\SV}$ as 
\begin{align*}
\min_{\BZ_\SU,\BZ_\SV}\la\P_\Omega\lb\G^*\lb\BZ_{\SU}\BZ_{\SV}^*\rb-\by\rb, \G^*\lb\BZ_{\SU}\BZ_{\SV}^*\rb-\by\ra
\quad\mbox{subject to}\quad(\I-\G\G^*)(\BZ_{\SU}\BZ_{\SV}^*) = \bzero,\numberthis\label{eq:constrained}\end{align*}
which is an equality constraint minimization problem.  Alternatively, \eqref{eq:constrained} can be interpreted as follows: we estimate  the rank $r$ matrix $\G\by$ by a Hankel matrix of the form $\BZ_{\SU}\BZ_{\SV}^*$ that minimizes the mismatch in the measurement domain. Once $\G\by$ is reconstructed, one can recover $\by$ via $\by=\G^*(\G\by)$.

Putting the constraint and the objective function in \eqref{eq:constrained} together allows us to  consider an optimization problem without the equality constraint  by minimizing
\begin{align*}
f(\BZ) = \ln(\I-\G\G^*)(\BZ_{\SU}\BZ_{\SV}^*)\rn_F^2+p^{-1}\la\P_\Omega(\G^*(\BZ_{\SU}\BZ_{\SV}^*)-\by),\G^*(\BZ_{\SU}\BZ_{\SV}^*)-\by\ra,
\end{align*}
where $$\BZ=\begin{bmatrix}\BZ_{\SU}\\\BZ_{\SV}\end{bmatrix}\in\C^{(n+1)\times r}$$ denotes the concatenation of $\BZ_{\SU}$ and $\BZ_{\SV}$, and the weight $p=m/n$ is  the sampling ratio.
Let $\G\by=\BU\BS\BV^*$ be the reduced singular value decomposition (SVD) of $\G\by$. Define \begin{align*}\BY=\begin{bmatrix}\BY_{\SU}\\\BY_{\SV}\end{bmatrix}\in\C^{(n+1)\times r},
\numberthis\label{eq:Y}
\end{align*} where $\BY_\SU=\BU\BS^{1/2}$ and $\BY_\SV=\BV\BS^{1/2}$. It is easily shown that $f(\BZ)=0$ and thus achieves its minimum for the set of matrices
\begin{align*}
\lcb \begin{bmatrix}\BY_{\SU}\BX\\\BY_{\SV}(\BX^{-1})^*\end{bmatrix},~\BX\in\C^{r\times r} \mbox{ is invertible}\rcb.\numberthis\label{eq:sol_set1}
\end{align*}
Note that \eqref{eq:sol_set1} is also a set of solutions for the equality constrained problem \eqref{eq:constrained}. Among this set of solutions, there are ones which are highly unbalanced, i.e., these having $\ln\BZ_\SU\rn_F\rightarrow 0$ and $\ln\BZ_\SV\rn_F\rightarrow \infty$, or vice versa. For example, let $\BZ_\SU=\alpha \BY_{\SU}$ and $\BZ_\SV=\alpha^{-1} \BY_{\SV}$ for $\alpha$ being a real number that approaches either zero or infinity. Those solutions are unfavorable for the purpose of both computation and analysis. In order to reduce the solution space and avoid the occurrence of the pathological solutions, we add the regularizer function 
\begin{align*}
g(\BZ) = \frac{1}{2}\ln\BZ_{\SU}^*\BZ_{\SU}-\BZ_{\SV}^*\BZ_{\SV}\rn_F^2
\end{align*}
to $f(\BZ)$ and instead consider the minimization problem with respect to
\begin{align*}
F(\BZ) = f(\BZ)+\lambda\cdot g(\BZ),\numberthis\label{eq:obj_F}
\end{align*}
where $\lambda>0$ is  to be determined. Here, $g(\BZ)$ in some sense penalizes the mismatch between the sizes of $\BZ_\SU$ and $\BZ_\SV$, and it was also used in  rectangular low rank  matrix recovery, see \cite{PF, FGD}.

Now, the set of solutions that minimizes $F(\BZ)$ or at which $F(\BZ)=0$ is given by 
\begin{align*}
\mathcal{S}=\lcb \begin{bmatrix}\BY_{\SU}\BR\\\BY_{\SV}\BR\end{bmatrix},~\BR\in\C^{r\times r} \mbox{ is unitary}\rcb.\numberthis\label{eq:sol_set2}
\end{align*}
The distance of a matrix $\BZ\in\C^{(n+1)\times r}$ to the solution set, denoted $\dist{\BZ}{\BY}$, is defined as
\begin{align*}
\dist{\BZ}{\BY}=\min_{\substack{\BR\BR^*=\BR^*\BR=\BI}}\ln\BZ-\BY\BR\rn_F.
\end{align*}
Let $\BY^*\BZ=\BQ_1\BLam\BQ_2^*$ be the SVD of $\BY^*\BZ$. 
 By the Von Neumann's trace inequality \cite{von_trace}, the above minimum is achieved at the unitary matrix $\BR_\SZ$ given by
\begin{align*}\BR_{\SZ}=\BQ_1\BQ_2^*.\numberthis\label{eq:R_form}
\end{align*}
\subsubsection{Which Feasible Set?} As we have already seen, the goal in spectrally sparse signal recovery is in fact to reconstruct a low rank Hankel matrix matrix $\G\by$ from its partial revealed skew-diagonals. In general, it is impossible  to reconstruct  a low rank matrix from entry-wise sampling unless its singular vectors are
weakly correlated with the sampling basis. Here, we are interested in $\mu_0$-incoherent matrix which was first introduced in \cite{candesrecht2009mc} for low rank matrix completion.
\begin{definition}\label{def:coh}
With $\G\by=\BU\BS\BV^*$ being the SVD of $\G\by$, we say $\G\by$ is $\mu_0$-incoherent   if there exists an absolute numerical constant $\mu_0>0$ such that 
\begin{align*}
\ln\BU\rn_{2,\infty} \leq \sqrt{\frac{\mu_0c_sr}{n}}\quad\mbox{and}\quad\ln\BV\rn_{2,\infty} \leq \sqrt{\frac{\mu_0c_sr}{n}},
\end{align*}
where $c_s=\max\{n/n_1,n/n_2\}$.
\end{definition}
A sufficient condition for $\G\by$ to be $\mu_0$-incoherent can be derived  based on the Vandermonde decomposition of $\G\by$. Assume that
\begin{align*}
\sigma_{\min}(\bm{E}_L^*\bm{E}_L)\geq \frac{n_1}{\mu_0},
\quad
\sigma_{\min}(\bm{E}_R^*\bm{E}_R)\geq \frac{n_2}{\mu_0}.\numberthis\label{eq:lower_l2}
\end{align*}
Then we have 
$$
\ln\BU^{(i,:)}\rn^2_2=\ln \be_i^*\BE_L(\BE_L^*\BE_L)^{-1/2}\rn_2^2\leq \ln \be_i^*\BE_L\rn_2^2\ln (\BE_L^*\BE_L)^{-1}\rn_2\leq\frac{\mu_0r}{n_1}\leq \frac{\mu_0c_sr}{n}
$$
and 
$$
\ln\BV^{(i,:)}\rn^2_2=\ln \be_i^*\BE_R(\BE_R^*\BE_R)^{-1/2}\rn_2^2\leq \ln \be_i^*\BE_R\rn_2^2\ln (\BE_R^*\BE_R)^{-1}\rn_2\leq\frac{\mu_0r}{n_2}\leq \frac{\mu_0c_sr}{n},
$$
which implies $\G\by$ is $\mu_0$-incoherent. Moreover, \cite[Thm.~2]{MUSIC} says that \eqref{eq:lower_l2} holds for undamping signals provided  the minimum wrap-around distance between each pair of the frequencies of the spectrally sparse signal is greater than about $2/n$.

Let $\mu$ and $\sigma$ be two numerical constants such that $\mu\geq \mu_0$ and $\sigma\geq \sigma_1(\G\by)$. When $\G\by$ is $\mu_0$-incoherent, the matrix $\BY$ constructed in \eqref{eq:Y} satisfies $\ln\BY\rn_{2,\infty}\leq \sqrt{\mu c_sr\sigma/n}$. Moreover, letting $\CS$ be a convex set defined as 
\begin{align*}
\CS = \lcb\BZ\in\C^{(n+1)\times r}~|~\ln\BZ\rn_{2,\infty}\leq \sqrt{\frac{\mu c_sr\sigma}{n}} \rcb,\numberthis\label{eq:set_C}
\end{align*}
it is evident that $\mathcal{S}\subset\CS$. Therefore,   we can restrict our search on the feasible set $\CS$ when computing the minimum or zero value of $F(\BZ)$.
\subsubsection{Algorithm}
The discussion above tells us that we can reconstruct the low rank factors $\BY_{\SU}$  and $\BY_{\SV}$ of the ground truth matrix $\G\by$ by minimizing the function $F(\BZ)$ on the feasible set $\CS$, namely
\begin{align*}
\min_{\BZ\in\CS} F(\BZ),\numberthis\label{eq:rec_meth}
\end{align*}
where $F(\BZ)$ is defined in \eqref{eq:obj_F} and $\CS$ is defined in \eqref{eq:set_C}.
We present a simple projected gradient descent algorithm for this problem, see Algorithm~\ref{alg:pgd}.
\begin{algorithm}[htp]
\caption{Projected Gradient Descent (PGD)}
\label{alg:pgd}
\begin{algorithmic} 
\Statex \textbf{Initialization:} $\bm{L}^{0}=p^{-1}\T_r(\G\P_\Omega(\by))=\BU^0\BS^0(\BV^0)^*$, $\TBZ^0=\begin{bmatrix}\BU^0(\BS^0)^{1/2}\\\BV^0(\BS^0)^{1/2}\end{bmatrix}$ and $\BZ^0=\PC{\TBZ^0}$.
\For{$k=0,1,\cdots$}\\
\quad1. ${\TBZ}^{k+1}=\BZ^k-\eta \nabla F(\BZ^k)$\\
\quad2. $\BZ^{k+1}=\PC{\TBZ^{k+1}}$
\EndFor
\Statex \textbf{Output:} $\BZ^k$ in the last iteration, $\by^k = \G^*(\BZ^k_\SU(\BZ^k_\SV)^*)$ and $\bx^k = \D^{-1}\by^k$.
\end{algorithmic}
\end{algorithm}
The algorithm consists of two phases: Initialization and gradient descent with a constant stepsize. The initial guess is computed via one-step hard thresholding, followed by projection onto the convex set $\CS$. The hard thresholding operator $\T_r(\cdot)$ returns the best rank $r$ approximation of a matrix, which can be computed via the partial SVD. Given a matrix $\BZ\in\C^{(n+1)\times r}$, the projection $\P_{\CS}(\BZ)$ can be computed by row-wise trimming,
\begin{align*}
[\P_{\CS}(\BZ)]^{(i,:)}=\begin{cases}\BZ^{(i,:)} & \mbox{if }\ln\BZ^{(i,:)}\rn_2\leq \sqrt{\frac{\mu c_sr\sigma}{n}},\\
\frac{\BZ^{(i,:)}}{\ln\BZ^{(i,:)}\rn_2}\sqrt{\frac{\mu c_sr\sigma}{n}}&\mbox{otherwise}.
\end{cases}
\end{align*}
In each iteration of the algorithm, the current estimate $\BZ^k$ is updated along the negative gradient descent direction $-\nabla F(\BZ^k)$, using a stepsize $\eta$, followed by projection onto the convex set $\CS$. Since we are working
with complex matrices, the gradient $F(\BZ)$ of a matrix $\BZ$ is calculated under the Wirtinger calculus, given by 
\begin{align*}
\nabla F(\BZ) =\begin{bmatrix}
\nabla F_{\SU}(\BZ)\\
\nabla F_{\SV}(\BZ)
\end{bmatrix}=
\begin{bmatrix}
\nabla f_{\SU}(\BZ)+\lambda\cdot\nabla g_{\SU}(\BZ) \\
\nabla f_{\SV}(\BZ)+\lambda\cdot\nabla g_{\SV}(\BZ)
\end{bmatrix},
\end{align*}
where
\begin{align*}
&\nabla f_{\SU}(\BZ) = \lb(\I-\G\G^*)(\BZ_{\SU}\BZ_{\SV}^*)\rb\BZ_{\SV} +p^{-1}\lb\G\P_\Omega(\G^*(\BZ_{\SU}\BZ_{\SV}^*)-\by)\rb\BZ_{\SV},\\
&\nabla f_{\SV}(\BZ) = \lb(\I-\G\G^*)(\BZ_{\SU}\BZ_{\SV}^*)\rb^*\BZ_{\SU} +p^{-1}\lb\G\P_\Omega(\G^*(\BZ_{\SU}\BZ_{\SV}^*)-\by)\rb^*\BZ_{\SU},\\
&\nabla g_{\SU}(\BZ) = \BZ_{\SU}(\BZ_{\SU}^*\BZ_{\SU}-\BZ_{\SV}^*\BZ_{\SV}),\\
&\nabla g_{{\SV}}(\BZ) = \BZ_{\SV}(\BZ_{\SV}^*\BZ_{\SV}-\BZ_{\SU}^*\BZ_{\SU}).
\end{align*}

PGD can be implemented very efficiently and the main computational cost per iteration is $O(r^2n+rn\log(n))$ flops, which lies in the computation of $\nabla F(\BZ)$ in each iteration. Taking the computation of $\nabla F_{\SU}(\BZ)$ as an example, we note that 
\begin{align*}
\nabla F_{\SU}(\BZ)=\G\lb p^{-1}\P_\Omega(\G^*(\BZ_{\SU}\BZ_{\SV}^*)-\by)-\G^*(\BZ_{\SU}\BZ_{\SV}^*)\rb\BZ_{\SV} + \BZ_{\SU}\lb\lambda\BZ_{\SU}^*\BZ_{\SU}+(1-\lambda)\BZ_{\SV}^*\BZ_{\SV}\rb.
\end{align*}
Clearly,  the second term can be computed using $O(r^2n)$ flops. Let $\bw =  p^{-1}\P_\Omega(\G^*(\BZ_{\SU}\BZ_{\SV}^*)-\by)-\G^*(\BZ_{\SU}\BZ_{\SV}^*)$. Since we can compute $\G^*(\BZ_{\SU}\BZ_{\SV}^*)$ by $r$ fast convolutions, $\bw$ can be obtained using $O(rn\log(n))$ flops. Moreover, $(\G\bw)\BZ_{\SV}$ can be computed via $r$ fast Hankel matrix-vector multiplications that also cost $O(rn\log(n))$ flops.

Before proceeding, it is worth noting that non-convex (projected) gradient decent methods have received intensive investigations for other low rank matrix recovery problems, such as unstructured low rank matrix recovery and matrix completion \cite{PF,FGD,FGD2}, phase retrieval \cite{candes2015phase,candes2017phase}, robust principle component analysis \cite{yudong1,yudong2}, and blind deconvolution \cite{sykw}.  In those papers, lower bounds on the sampling complexity have been established under different random measurement models, showing that the number of  measurements needed for the successful recovery of the target matrices is essentially determined by the number of  degrees of freedom in the matrices. In particular, 
a projected gradient descent algorithm was studied in \cite{FGD} for unstructured rectangular  low rank matrix completion. 
The convergence analysis of PGD in this paper is directly inspired by \cite{FGD}, though the technical details are substantially different.
\subsection{Main Result}
Let $\Omega=\{a_k~|~k=1,\cdots,m\}$. We consider the sampling with replacement model in this paper, where each index 
$a_k$ is drawn independently and uniformly from $\{0,\cdots,n-1\}$. Under this sampling model, for a vector $\bz\in\C^n$, the projection $\P_\Omega(\bz)$ is given by
\begin{align*}
\P_\Omega(\bz) =\sum_{k=1}^mz_{a_k}\be_{a_k},\numberthis\label{eq:tmp110}
\end{align*}
and for two vectors $\bz,~\bw\in\C^n$, the inner product $\la\P_\Omega(\bz),\bw\ra$ is given by
\begin{align*}
\la\P_\Omega(\bz),\bw\ra = \sum_{k=1}^m\bar{z}_{a_k}w_{a_k}\numberthis\label{eq:tmp112}.
\end{align*}
In the guarantee analysis of PGD, we assume $\mu$ and $\sigma$ in \eqref{eq:set_C} are two tuning parameters obeying $\mu\geq \mu_0$ and $\sigma\geq \sigma_1(\G\by)$ so that  $\BY\in\CS$. For conciseness, we  take  $\sigma=\sigma_1(\BL_0)/(1-\varepsilon_0)$ for some $0<\varepsilon_0<1$ and will later show that $\sigma\geq \sigma_1(\G\by)$ with high probability. 

\begin{theorem}[Exact Recovery]\label{thm:exact_recovery} Assume $\G\by$ is $\mu_0$-incoherent. Let $\varepsilon_0$ be a absolute constant obeying $0<\varepsilon_0\leq 1/11$. Let $\mu\geq \mu_0$ and $\sigma=\sigma_1(\BL_0)/(1-\varepsilon_0)$. If we take $\lambda=1/4$ in \eqref{eq:obj_F}, then with probability at least $1-c_1\cdot n^{-2}$, the sequence $\lcb\BZ^k\rcb_{k\geq 1}$ returned by Algorithm~\ref{alg:pgd} obeys 
\begin{align*}
\distsq{\BZ^k}{\BY}\leq (1-\eta\nu)^k\distsq{\BZ^0}{\BY}
\end{align*}
for 
\begin{align*}
\eta\leq \frac{\sigma_r(\G\by)}{600(\mu c_sr)^2\sigma_1^2(\G\by)}\quad\mbox{and}\quad \nu=\frac{1}{10}\sigma_r(\G\by)
\end{align*}
provided $m\geq c_2\hspace{0.05cm}\varepsilon_0^{-2}\mu^2c_s^2\kappa^2r^2\log(n)$, where $\kappa=\sigma_1(\G\by)/\sigma_r(\G\by)$.
\end{theorem}
\begin{remark*}{\normalfont
1). After an approximation of $\G\by$, given by $\BZ^k_\SU(\BZ^k_\SV)^*$, is obtained from PGD, we can estimate $\by$ by $\by^k= \G^*(\BZ^k_\SU(\BZ^k_\SV)^*)$, and in turn estimate $\bx$ by $\D^{-1}\by^k$. Recall from \eqref{eq:R_form} that $\BQ_{\SZ^k}$ is a unitary matrix which obeys $\dist{\BZ^k}{\BM}=\ln\BZ^k-\BM\BQ_{\SZ^k}\rn_F$.
A simple calculation yields
{\begin{align*}
\ln\bx^k-\bx\rn_2&\leq\ln\by^k-\by\rn_2 = \ln \G^*(\BZ^k_\SU(\BZ^k_\SV)^*)-\G^*(\G\by)\rn_2\\
&\leq \ln \BZ^k_\SU(\BZ^k_\SV)^*-\BY_\SU\BY_\SV^*\rn_F\leq \frac{1}{\sqrt{2}}\ln\BZ^k(\BZ^k)^*-\BM\BM^*\rn_F\\
&= \frac{1}{\sqrt{2}}\ln \BZ^k(\BZ^k-\BM\BQ_{\SZ^k})^*+(\BZ^k-\BM\BQ_{\SZ^k})(\BM\BQ_{\SZ^k})^*\rn_F\\
&\leq \frac{1}{\sqrt{2}}\lb\|\BZ^k\|_2+\ln\BM\rn_2\rb\dist{\BZ^k}{\BM}\rightarrow 0, \quad\mbox{as }\dist{\BZ^k}{\BM}\rightarrow 0.
\end{align*}}

2). After each iteration, Theorem~\ref{thm:exact_recovery} implies that the distance between the estimate given by PGD and $\BM$ is reduced by at least of a factor of $1-O(1/(\mu c_sr\kappa)^2)$. Thus, after $k\approx O((\mu c_sr\kappa)^2\log(1/\epsilon))$ iterations, one has $\distsq{\BZ^k}{\BY}\leq \epsilon\cdot\distsq{\BZ^0}{\BY}$.

3). It was shown in \cite{FIHT} that FIHT can achieve exact recovery when the number of revealed entries is of order $O(\kappa^6r^2\log^2(n))$. In contrast, the sampling complexity of PGD is only a quadratic function of $\kappa$ and a linear function of $\log(n)$. Moreover, the exact recovery guarantee of FIHT relies on a more complicated initialization scheme which requires  a partition of the observed entries into $O(\log(n))$ groups, while the initial guess constructed  for the exact recovery guarantee of PGD can be computed much more easily.
}\end{remark*}
\subsection{Extension to Higher Dimension}
So far we have restricted our attention to one-dimensional spectrally sparse signal reconstruction problem. Our algorithm and results can be extended to higher dimensions based on the Hankel structures of multi-dimensional
spectrally sparse signals. Without loss of generality, we discuss the two-dimensional setting but emphasize that the situation in general $d$-dimensions is similar.

Let $w_k=e^{(2\pi\imath f_{1k}-\tau_{1k})}$ and $z_k=e^{(2\pi\imath f_{2k}-\tau_{2k})}$ for $r$ frequency pairs $(f_{1k}, f_{2k})\in[0,1)^2$ and $r$ damping factor pairs $(\tau_{1k}, \tau_{2k})\in\R^2_+$. A two-dimensional spectrally sparse array $\BX\in\C^{N_1\times N_2}$  can be expressed as 
\begin{align*}
\BX^{(a,b)} = \sum_{k=1}^r d_kw_k^{a}z_k^{b}, \quad (a,b)\in[N_1] \times [N_2].
\end{align*}
The two-fold Hankel matrix of $\BX$ is given by 
\begin{align*}
\H\BX =\begin{bmatrix}
\H\BX^{(:,0)} & \H\BX^{(:,1) }& \H\BX^{(:,2)} & \cdots & \cdots& \H\BX^{(:,N_2-n_2)}\\
\H\BX^{(:,1) } & \H\BX^{(:,2)} &\cdots& \cdots & \cdots &  \H\BX^{(:,N_2-n_2+1)}\\
 \H\BX^{(:,2)} & \cdots& \cdots & \cdots & \cdots & \H\BX^{(:,N_2-n_2+2)}\\
\vdots & \vdots & \vdots & \vdots & \vdots &\vdots\\
\H\BX^{(:,n_2-1)} & \H\BX^{(:,n_2)} & \cdots & \cdots &  \cdots  &\H\BX^{(:,N_2-1)}
\end{bmatrix},
\end{align*}
where each block is an $n_1\times (N_1-n_1+1)$ Hankel matrix corresponding to a column of $\BX$, 
\begin{align*}
\H\BX^{(:,b)} = \begin{bmatrix}
\H\BX^{(0,b)}  & \H\BX^{(1,b)} & \H\BX^{(2,b)} & \cdots & \cdots& \H\BX^{(N_1-n_1,b)}\\
\H\BX^{(1,b)} &\H\BX^{(2,b)} &\cdots& \cdots & \cdots &  \H\BX^{(N_1-n_1+1,b)}\\
\H\BX^{(2,b)} & \cdots& \cdots & \cdots & \cdots & \H\BX^{(N_1-n_1+2,b)}\\
\vdots & \vdots & \vdots & \vdots & \vdots &\vdots\\
\H\BX^{(n_1-1,b)} & \H\BX^{(n_1,b)} & \cdots & \cdots &  \cdots  &\H\BX^{(N_1-1,b)}
\end{bmatrix}.
\end{align*}
Clearly, $\H\BX$ is an $(n_1n_2)\times (N_1-n_1+1)(N_2-n_2+1)$ matrix. Letting $i = i_1 + i_2\cdot n_1$ and $j  = j_1 + j_2 \cdot (N_1-n_1+1)$, the $(i,j)$-th entry of $\H\BX$ is given by 
\begin{align*}
\H\BX^{(i,j)} = \BX^{(i_1+j_1,i_2+j_2)} =\sum_{k=1}^r d_k\lb w_k^{i_1}z_k^{i_2}\rb\lb w_k^{j_1}z_k^{j_2}\rb.\numberthis\label{eq:2d_signal_dec}
\end{align*} 
For $k=1,\cdots,r$, we define  the four vectors $\bw_k^{[n_1]}$, $\bw_k^{[N_1-n_1+1]}$, $\bz_k^{[n_2]}$, and $\bz_k^{[N_2-n_2+1]}$ as 
\begin{align*}
\bw_k^{[n_1]}=\begin{bmatrix}1\\w_k\\\vdots\\w_k^{n_1-1}\end{bmatrix},\quad
\bw_k^{[N_1-n_1+1]}=\begin{bmatrix}1\\w_k\\\vdots\\w_k^{N_1-n_1}\end{bmatrix},\quad
\bz_k^{[n_2]}=\begin{bmatrix}1\\z_k\\\vdots\\z_k^{n_2-1}\end{bmatrix},\quad\mbox{and }
\bz_k^{[N_2-n_2+1]}=\begin{bmatrix}1\\z_k\\\vdots\\z_k^{N_2-n_2}\end{bmatrix}.
\end{align*}
Let $\BE_L$ be an $(n_1n_2)\times r$ matrix with the $k$-th column being given by \kw{$\bz_k^{[n_2]}\otimes \bw_k^{[n_1]}$}, and let $\BE_R$ be an $(N_1-n_1+1)(N_2-n_2+1)\times r$ matrix with the $k$-th column being given by \kw{$\bz_k^{[N_2-n_2+1]}\otimes \bw_k^{[N_1-n_1+1]}$}. Then it follows from \eqref{eq:2d_signal_dec} that $\H\BX$ admits the Vandermonde decomposition 
\begin{align*}
\H\BX = \BE_L\BD\BE_R^T,
\end{align*}
where $\BD=\diag(d_1,\cdots,d_r)$. Thus, it is self-evident that $\H\BX$ is  a rank $r$ matrix. 

As in the one-dimensional case, the goal in two-dimensional spectral sparse signal reconstruction is to reconstruct $\BX$ from the partial revealed entries of $\BX$, denoted $\P_\Omega(\BX)$, where $\Omega$ is a subset of $[N_1]\times [N_2]$.
\kw{
Let $w_a$ be the number of entires in the $a$-th skew-diagonal of an $n_1\times (N_1-n_1+1)$ matrix, and let $w_b$ be the number of entires in the $b$-th skew-diagonal of an $n_2\times (N_2-n_2+1)$ matrix. Define $\D$ as a linear operator from  $\C^{N_1\times N_2}$ to $\C^{N_1\times N_2}$ which scales the $(a,b)$-th entry of an $N_1\times N_2$ matrix by $\sqrt{w_aw_b}$. 
}After the change of variables $\BYY=\D\BX$ and $\G=\H\D^{-1}$, we can instead consider the recovery of $\BYY$ from $\P_\Omega(\BYY)$, which is equivalent to a low rank Hankel matrix completion problem since $\G\BYY=\H\BX$ is rank $r$.
Following the route set up in Section~\ref{sec:alg_pgd}, this task can be attempted by minimizing 
\begin{align*}
F(\BZ) &=  \ln(\I-\G\G^*)(\BZ_{\SU}\BZ_{\SV}^*)\rn_F^2+p^{-1}\la\P_\Omega(\G^*(\BZ_{\SU}\BZ_{\SV}^*)-\BYY),\G^*(\BZ_{\SU}\BZ_{\SV}^*)-\BYY\ra\\&\quad+\frac{\lambda}{2}\ln\BZ_{\SU}^*\BZ_{\SU}-\BZ_{\SV}^*\BZ_{\SV}\rn_F^2
\end{align*}
subject to a feasible set $\CS$, where $\G^*$ is the adjoint of $\G$ which obeys $\G^*\G=\I$, \kw{$$\BZ=\begin{bmatrix}\BZ_\SU\\\BZ_\SV\end{bmatrix}$$ is an $(n_1n_2+(N_1-n_1+1)(N_2-n_2+1))\times r$} matrix, and $\CS$ is a convex set  similar to the one defined in \eqref{eq:set_C} but  the size of $\BZ$ is different.

Therefore, a projected gradient descent algorithm can also be  developed for the two-dimensional spectrally sparse signal reconstruction problem.  Let $\G\BYY=\BU\BS\BV^T$ be the SVD of $\G\BYY$. We say $\G\BYY$ is $\mu_0$-incoherent if there exists a numerical constant $\mu_0>0$ such that 
\kw{\begin{align*}
\ln\BU\rn_{2,\infty} \leq \sqrt{\frac{\mu_0c_sr}{N_1N_2}}\quad\mbox{and}\quad\ln\BV\rn_{2,\infty} \leq \sqrt{\frac{\mu_0c_sr}{N_1N_2}},
\end{align*}}where $c_s=\max\{N_1N_2/(n_1n_2),N_1N_2/((N_1-n_1+1)(N_2-n_2+1))\}$. Based on \cite[Theorem~1]{MUSIC2D}, one can  show that $\G\BYY$ ($=\H\BX$) is $\mu_0$-incoherent if there is no damping in $\BX$ and the minimum wrap-around distance between the underlying frequencies $\{f_{ik}\}_ {k=1}^r$ is greater than about ${2}/{N_i}$ for $i=1, 2$.
Let $$\BM=\begin{bmatrix} \BM_\SU\\\BM_\SV\end{bmatrix},$$ where $\BM_\SU=\BU\BS^{1/2}$ and $\BM_{\SV}=\BV\BS^{1/2}$.
 If we assume $\G\BYY$ is $\mu_0$-incoherent and $\mu$ and $\sigma$ in $\CS$ are properly tuned such that $\BM\in\CS$, then the exact guarantee analysis of PGD for the one-dimensional case can be extended immediately to the two-dimensional case. It can be  established that \kw{$O(\mu^2c_s^2\kappa^2r^2\log(N_1N_2))$} number of measurements are sufficient for PGD to achieve the successful recovery of a two-dimensional spectrally sparse signal.

%% file: numerics.tex
\section{Numerical Experiments}\label{sec:numerics}

In this section, we conduct numerical experiments to evaluate the performance of PGD\footnote{In our random simulations, we didn't find much difference between the performance of PGD and the performance of the gradient descent algorithm applied  to $f(\BZ)$ directly. However, since the extra cost incurred by computing the gradient of $g(\BZ)$ and the projection $\P_\CS(\BZ)$ is marginal, it is appealing  to run PGD for its recovery guarantee.}. The experiments are executed from MATLAB R2017a on a 64-bit Linux machine with multi-core Intel Xeon CPU E5-2667 v3 at 3.20GHz and 64GB of RAM.  
In Section~\ref{sec:phase}, we investigate the largest number of Fourier components that can be successfully recovered by PGD. The tests are conducted  on one-dimensional signals in large part due to the high computational cost of this type of simulations. Then we evaluate PGD against  computational efficiency, robustness to additive noise, and sensitivity to mis-specification of model order on three-dimensional signals in Sections~\ref{sec:nu_speed}, \ref{sec:robustness}, and \ref{sec:sensitivity}, respectively.
The initial guess of  PGD is computed using the PROPACK package \cite{PROPACK}, and the parameters $\mu$ and $\sigma$ used in the projection are estimated from the initialization. Instead of using the constant stepsize suggested in the main result  which appears to  be conservative, we choose the stepsize via a backtracking line search in the implementation. 

\subsection{Empirical Phase Transition}\label{sec:phase}

We evaluate the  recovery ability of PGD in the framework of phase transition and compare it with ANM \cite{Tang}, EMaC \cite{Chi} and FIHT \cite{FIHT}. ANM and EMaC are implemented using CVX \cite{CVX} with default parameters. The test spectrally sparse signals of length $n$ with $r$ frequency components are formed in the following way: each frequency $f_k$ is randomly generated from $[0,1)$, and the argument of each complex coefficient $d_k$ is uniformly sampled from $[0,2\pi)$ while the amplitude is selected to be $1+10^{0.5c_k}$ with $c_k$ being uniformly distributed on $[0,1]$. We test two different settings for the frequencies: a)  no separation condition is imposed  on $\{f_k\}_{k=1}^r$, and b)  the wrap-around distances between each pair of  the randomly drawn frequencies are guaranteed to be greater than $1.5/n$. 
After a signal is formed, $m$ of its entries are sampled uniformly at random. For a given triple $(n,r,m)$, $50$ random tests are conducted. We consider an algorithm to have successfully reconstructed a test signal if the root mean squared error (RMSE) is less than $10^{-3}$, $$\|\bm{x}_{rec}-\bm{x}\|_2/\|\bm{x}\|\leq 10^{-3}.$$ The tests are conducted with $n=127$ and $p=m/n$ taking 18 equispaced values from 0.1 to 0.95.  For a fixed pair of $(n,m)$, we start with $r=1$ and then increase the value of $r$ by one until it reaches a value such that the tested algorithm fails all the $50$ random tests.  FIHT is terminated when $\|\bx^{k+1}-\bx^{k}\|_2/\|\bx^{k}\|_2\leq 10^{-6}$ or a maximum number of  iteration is reached. PGD is terminated when one of the following three conditions is met: $\|\bx^{k+1}-\bx^{k}\|_2/\|\bx^{k}\|_2\leq 10^{-7}$, $|F(\tilde{\BZ}^{k+1})-F(\BZ^k)|/F(\BZ^k)\leq 10^{-5}$, or a maximum number of  iteration is reached.

We plot in Figure~\ref{fig:PT-Curves} the empirical recovery phase transition curves that identify the 80\% success rate for each tested algorithm under the two different frequency settings. When the  frequencies are separated by at least $1.5/n$, the right plot shows that ANM has the highest phase transition curve, and the phase transition curve of PGD closely tracks that of ANM. The performance of ANM degrades severely when there is no frequency separation requirement. In both of the frequency settings, the recovery phase transition curves of PGD are overall higher than that of EMaC. In the region of greatest interest where $p\leq 0.5$, the recovery phase transition curves of PGD are substantially higher than that of FIHT.

\begin{figure}[!htb]
\centering
	\includegraphics[width=0.45 \textwidth]{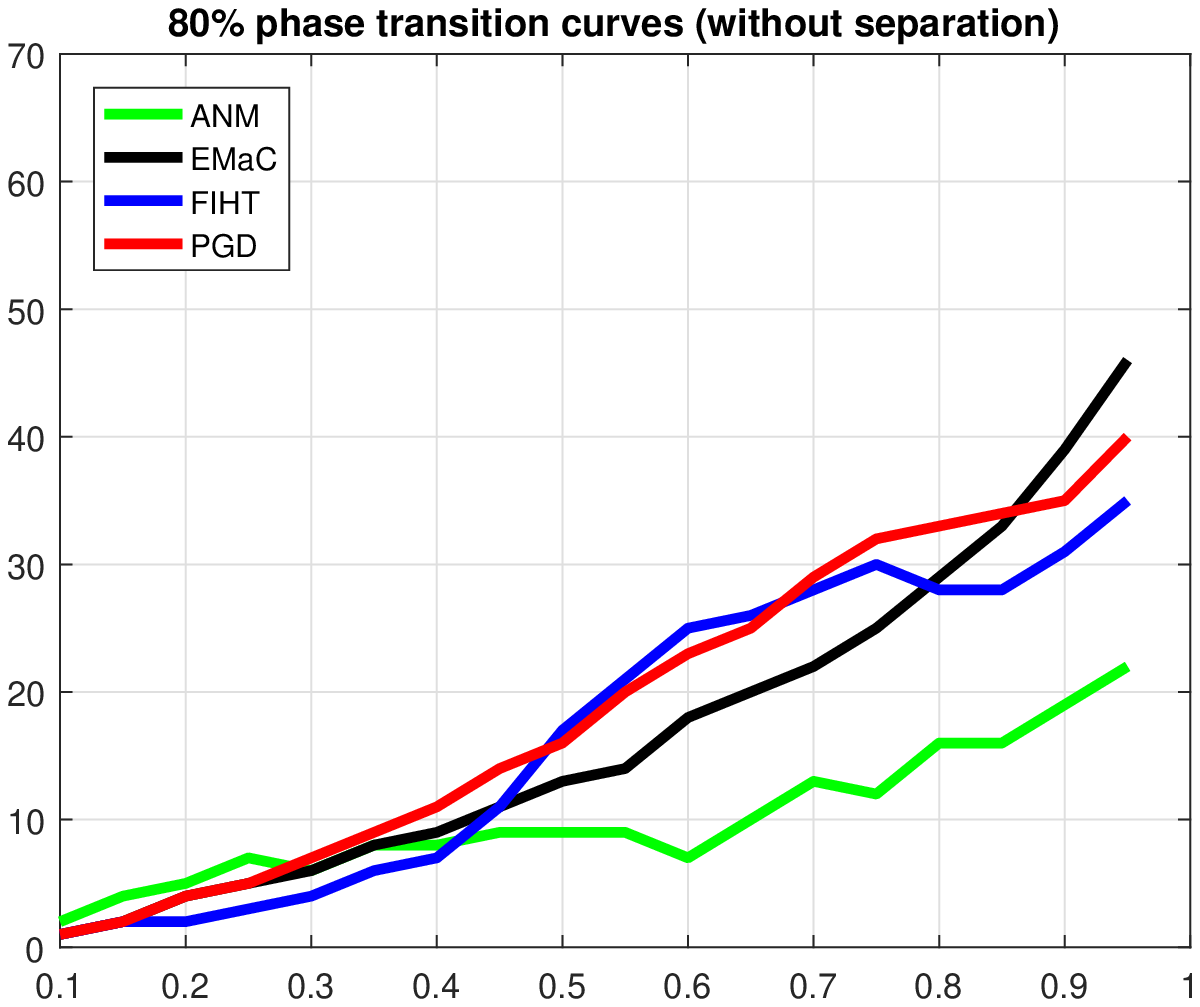}
	\includegraphics[width=0.45 \textwidth]{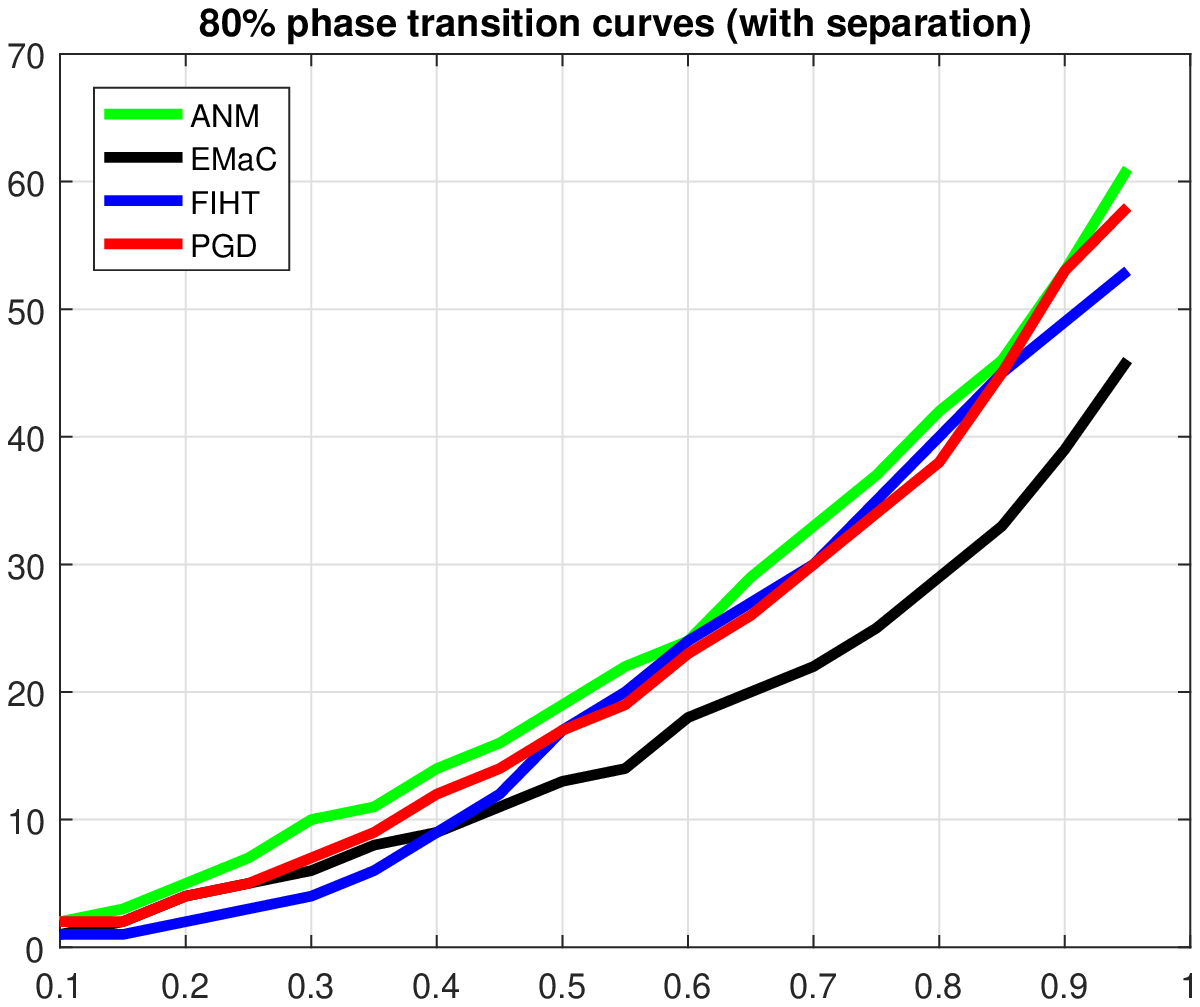}
\caption{$80\%$ phase transition curves: $x$-axis is $p=m/n$ and $y$-axis is $r$. Left:  signals  are formed by random frequencies without separation enforcement; Right:  signals are formed by random frequencies separated by at least $1.5/n$.}
\label{fig:PT-Curves}
\end{figure}
 
\subsection{Computational Efficiency}\label{sec:nu_speed}
\begin{table}[!htb]
\caption{{Average SR, RMSE, ITER and TIME values of FIHT and PGD over $10$ random problem instances in the undamped case with $m\approx 130\log(n)$.}}\label{table:efficiency_undamped}
\begin{center}
\makegapedcells
\setcellgapes{3pt}
\small
\begin{tabular}{ccccccccc}
\hline
\multicolumn{1}{|c|}{$r$} & \multicolumn{4}{c|}{20} & \multicolumn{4}{c|}{30} \\
\hline
\multicolumn{1}{|c|}{} & \multicolumn{1}{c}{SR} & \multicolumn{1}{c}{RMSE} & \multicolumn{1}{c}{ITER} & \multicolumn{1}{c|}{TIME (s)} & \multicolumn{1}{c}{SR} & \multicolumn{1}{c}{RMSE} & \multicolumn{1}{c}{ITER} & \multicolumn{1}{c|}{TIME (s)} \\
\hline
\multicolumn{1}{|c|}{} & \multicolumn{8}{c|}{with separation} \\
\hline
\multicolumn{1}{|c|}{FIHT} & 1 & 3.6e-4 & 18.7 & \multicolumn{1}{c|}{256} & 0.5 & 4.8e-4 & 123 & \multicolumn{1}{c|}{2278} \\
\hline
\multicolumn{1}{|c|}{PGD} & 1 & 1.4e-4 & 33.6 & \multicolumn{1}{c|}{490} & 1 & 2.7e-4 & 48.3 & \multicolumn{1}{c|}{1049} \\
\hline
\multicolumn{1}{|c|}{} & \multicolumn{8}{c|}{without separation} \\
\hline
\multicolumn{1}{|c|}{FIHT} & 1 & 3.5e-4 & 18.6 & \multicolumn{1}{c|}{250} & 0.2 & 4.8e-4 & 66.5 & \multicolumn{1}{c|}{1275} \\
\hline
\multicolumn{1}{|c|}{PGD} & 1 & 1.7e-4 & 33.6 & \multicolumn{1}{c|}{492} & 1 & 3.0e-4 & 54.6 & \multicolumn{1}{c|}{1186} \\
\hline
\end{tabular}
\end{center}
\end{table} 

\begin{table}[!htb]
\caption{{Average SR, RMSE, ITER and TIME values of FIHT and PGD over $10$ random problem instances in the damped case with $m\approx 0.03n$.}}\label{table:efficiency_damped}
\begin{center}
\makegapedcells
\setcellgapes{3pt}
\small
\begin{tabular}{ccccccccc}
\hline
\multicolumn{1}{|c|}{$r$} & \multicolumn{4}{c|}{20} & \multicolumn{4}{c|}{30} \\
\hline
\multicolumn{1}{|c|}{} & \multicolumn{1}{c}{SR} & \multicolumn{1}{c}{RMSE} & \multicolumn{1}{c}{ITER} & \multicolumn{1}{c|}{TIME (s)} & \multicolumn{1}{c}{SR} & \multicolumn{1}{c}{RMSE} & \multicolumn{1}{c}{ITER} & \multicolumn{1}{c|}{TIME (s)} \\
\hline
\multicolumn{1}{|c|}{} & \multicolumn{8}{c|}{with separation} \\
\hline
\multicolumn{1}{|c|}{FIHT} & 1 & 2.9e-4 & 12.7  & \multicolumn{1}{c|}{170} & 0.2  & 3.2e-4 & 16.5 & \multicolumn{1}{c|}{321} \\
\hline
\multicolumn{1}{|c|}{PGD} & 1 & 3.3e-4  & 21.8  & \multicolumn{1}{c|}{321} & 1 & 4.8e-4 & 41.5  & \multicolumn{1}{c|}{1028} \\
\hline
\multicolumn{1}{|c|}{} & \multicolumn{8}{c|}{without separation} \\
\hline
\multicolumn{1}{|c|}{FIHT} & 1 & 2.4e-4  & 10.9  & \multicolumn{1}{c|}{152} & 0.1 & 4.1e-4  & 16  & \multicolumn{1}{c|}{325} \\
\hline
\multicolumn{1}{|c|}{PGD} & 1 & 2.6e-4  & 17.4  & \multicolumn{1}{c|}{258} & 1  & 4.5e-4  & 37.4  & \multicolumn{1}{c|}{863} \\
\hline
\end{tabular}
\end{center}
\end{table} 

PGD has the same leading-order computational complexity as FIHT, and both of them are able to handle large and high-dimensional signals. We compare the computational performance of these two algorithms on undamped and damped three-dimensional spectrally sparse signals 
of size $n=64\times 128\times 512$. 
Tests are conducted with $r\in\{20,30\}$ and $m\approx 130\log(n)$ in the undamped setting while $m\approx 0.03 n$ in the damped setting, and we test signals which obey the frequency separation condition as well as signals which are fully random. 
As to the damping factors, for $1\leq k\leq r$, $1/\tau_{1k}$ is uniformly sampled from $[8~16]$, $1/\tau_{2k}$ is uniformly sampled from $[16~32]$, and $1/\tau_{3k}$ is uniformly sampled from $[64~128]$. 
For each triple of $(r, \mbox{undamped/damped, with/without separation})$, $10$ random problem instances are tested. FIHT is terminated when $\|\bx^{k+1}-\bx^{k}\|_2/\|\bx^{k}\|_2\leq 10^{-3}$ or $\|\bx^{k+1}-\bx^{k}\|_2/\|\bx^{k}\|_2\geq 2$ which usually implies divergence. PGD is terminated when $\|\bx^{k+1}-\bx^{k}\|_2/\|\bx^{k}\|_2\leq 2\times 10^{-4}$.
The average computational time (referred to as TIME) and average number of iterations (referred to as ITER) of FIHT and PGD over tests of successful recovery are summarized in Tables \ref{table:efficiency_undamped} and \ref{table:efficiency_damped} for the  undamped and damped signals, respectively. For the sake of completeness, we also include the ratio of successful recovery out of the 10 random tests  (referred to as SR)  for each algorithm 
 in the tables.
 
First it is worth noting that PGD succeeded in  all the $10$ random tests under each test setting when $r=30$, whereas FIHT  only succeeded in a small fraction of the  tests. Thus,  Tables \ref{table:efficiency_undamped} and \ref{table:efficiency_damped} show that PGD is able to more reliably recover signals that consist of   a larger number of Fourier components, which  coincides with our observations on one-dimensional signals in Section~\ref{sec:phase}. The tables also show that  FIHT requires fewer number of iterations and less computational time than PGD to achieve  convergence for  easier problem instances when $r=20$, while PGD is faster when $r=30$ and the test signals are undamped.

\subsection{Robustness to Additive Noise}\label{sec:robustness}

We demonstrate the performance of PGD under additive noise by conducting tests on 3D signals of the same size as in Section \ref{sec:nu_speed} but with measurements corrupted by the vector
$$
\be=\theta \cdot \|\mathcal{P}_{\Omega}(\bm{x})\|_2 \cdot \frac{\bm{w}}{\|\bm{w}\|_2},
$$
where $\bx$ is a reshaped three-dimensional spectrally sparse signal to be reconstructed, the entries of $\bm{w}$ are i.i.d. standard complex Gaussian random variables, and $\theta$ is referred to as the noise level. 

Tests are conducted with $7$ different values of $\theta$ from $10^{-3}$ to 1, corresponding to $7$ equispaced signal-to-noise ratios (SNR) from 60 to 0 dB. For each value of $\theta$, 10 random instances are tested. PGD is terminated when $\|\bx^{k+1}-\bx^{k}\|_2/\|\bx^{k}\|_2\leq 10^{-5}$. In our simulations, we fix $r=20$ and choose $m\in\{130\log(n),195\log(n)\}$ in the undamped setting while  $m\in\{0.03n,0.045n\}$ in the damped setting. The frequencies of the test signals are randomly generated from $[0,1)$ without the separation requirement and the damping factors are generated in the same fashion as in Section \ref{sec:nu_speed}. The average RMSE of the  reconstructed signals (measured in negative dB) plotted against the input SNR values of the samples is presented in Figure \ref{fig:robustness}.  The plots display a desirable linear scaling between the relative reconstruction error and the noise level for both the undamped and damped signals.  Moreover, the relative reconstruction error decreases linearly on a log-log scale as the number of measurements increases.

\begin{figure}[!htb]
\centering
	\includegraphics[width=0.45 \textwidth]{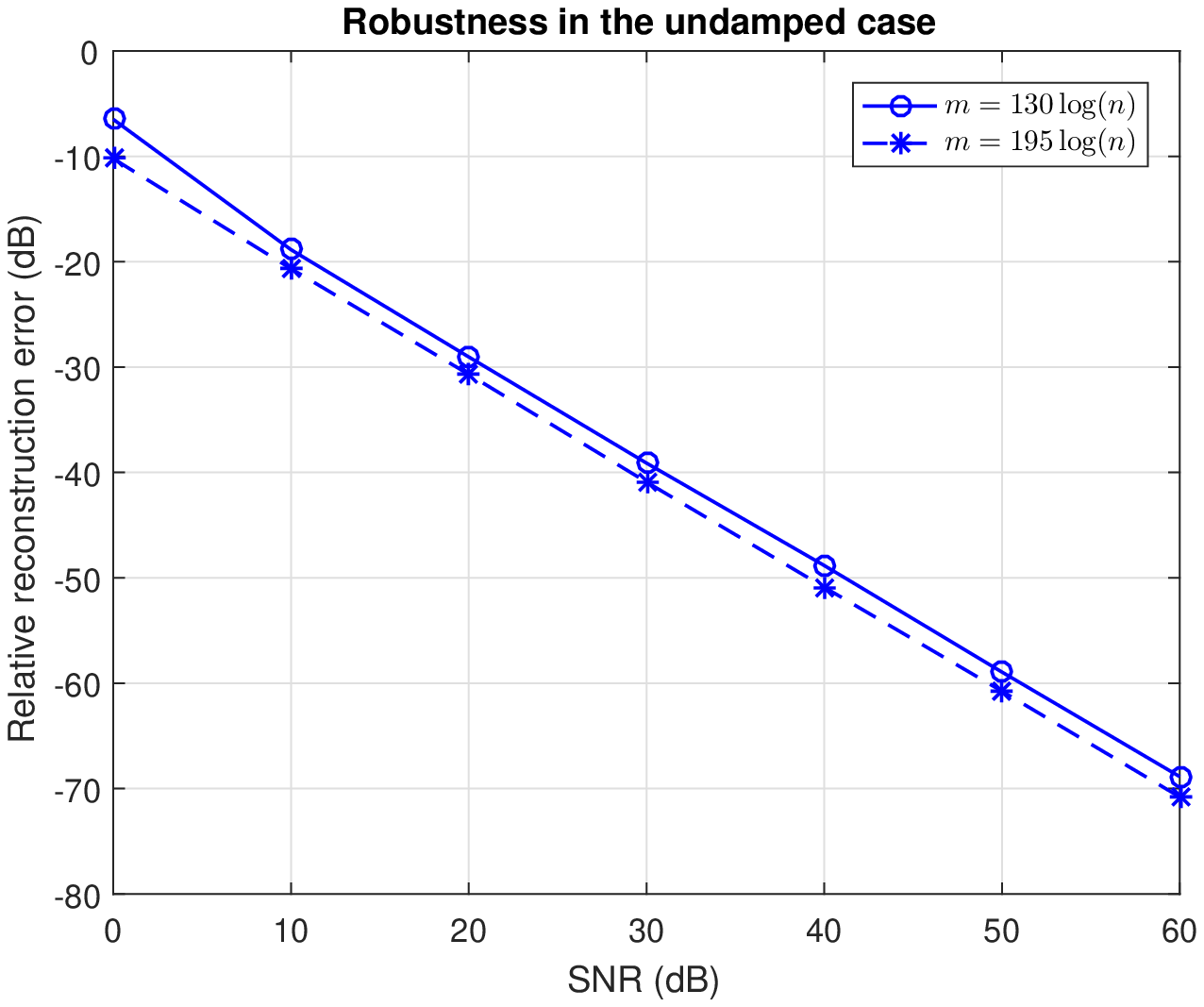}
	\includegraphics[width=0.45 \textwidth]{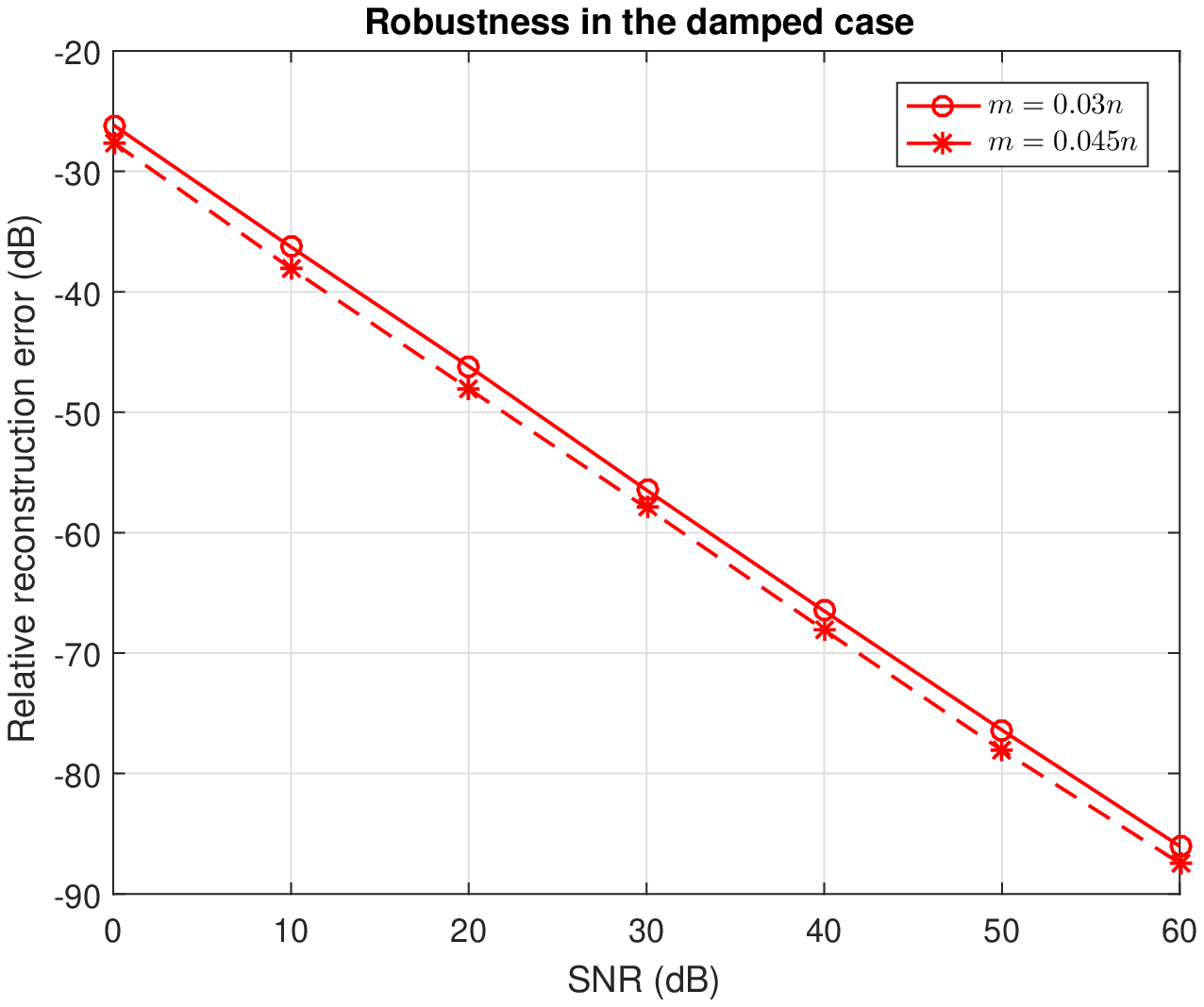}
\caption{Performance of PGD under additive noise. Left: no damping in the test signals; Right:  signals are generated with damping.}
\label{fig:robustness}
\end{figure}

\subsection{Sensitivity to Model Order}\label{sec:sensitivity}

\begin{table}[!htb]
\caption{Median values of ITER and SNR over 10 random problem instances  with $5\leq r\leq 40$ and $\mbox{SNR}\in\{\infty, 20, 0\}$ for undamped signals. The true model order is $r=20$.}\label{table:sensitivity_undamped}
\begin{center}
\makegapedcells
\setcellgapes{3pt}
\small
\begin{tabular}{|ccccccccc}
\hline
\multicolumn{1}{|c|}{Test Rank} & \multicolumn{1}{c|}{$5$} & \multicolumn{1}{c|}{$10$} & \multicolumn{1}{c|}{$15$} & \multicolumn{1}{c|}{$20$} & \multicolumn{1}{c|}{$25$} & \multicolumn{1}{c|}{$30$} & \multicolumn{1}{c|}{$35$} & \multicolumn{1}{c|}{$40$} \\
\hline
\multicolumn{1}{|c|}{} & \multicolumn{8}{c|}{SNR$=\infty$} \\
\hline
\multicolumn{1}{|c|}{ITER} & \multicolumn{1}{c|}{18.5} & \multicolumn{1}{c|}{18.5} & \multicolumn{1}{c|}{23.5} & \multicolumn{1}{c|}{45} & \multicolumn{1}{c|}{798} & \multicolumn{1}{c|}{1047} & \multicolumn{1}{c|}{1209} & \multicolumn{1}{c|}{1343} \\
\hline 
\multicolumn{1}{|c|}{SNR} & \multicolumn{1}{c|}{2.093} & \multicolumn{1}{c|}{4.844} & \multicolumn{1}{c|}{8.293} & \multicolumn{1}{c|}{99.63} & \multicolumn{1}{c|}{69.43} & \multicolumn{1}{c|}{67.08} & \multicolumn{1}{c|}{65.00} & \multicolumn{1}{c|}{63.72} \\ 
\hline
\multicolumn{1}{|c|}{} & \multicolumn{8}{c|}{SNR$=20$} \\
\hline
\multicolumn{1}{|c|}{ITER} & \multicolumn{1}{c|}{17.5} & \multicolumn{1}{c|}{23} & \multicolumn{1}{c|}{28.5} & \multicolumn{1}{c|}{40.5} & \multicolumn{1}{c|}{1524} & \multicolumn{1}{c|}{1969} & \multicolumn{1}{c|}{1964} & \multicolumn{1}{c|}{2514} \\ 
\hline
\multicolumn{1}{|c|}{SNR} & \multicolumn{1}{c|}{2.040} & \multicolumn{1}{c|}{4.848} & \multicolumn{1}{c|}{8.277} & \multicolumn{1}{c|}{29.05} & \multicolumn{1}{c|}{26.70} & \multicolumn{1}{c|}{25.58} & \multicolumn{1}{c|}{24.55} & \multicolumn{1}{c|}{23.75} \\ 
\hline
\multicolumn{1}{|c|}{} & \multicolumn{8}{c|}{SNR$=0$} \\
\hline
\multicolumn{1}{|c|}{ITER} & \multicolumn{1}{c|}{19.5} & \multicolumn{1}{c|}{25} & \multicolumn{1}{c|}{218.5} & \multicolumn{1}{c|}{427.5} & \multicolumn{1}{c|}{589} & \multicolumn{1}{c|}{569.5} & \multicolumn{1}{c|}{638} & \multicolumn{1}{c|}{787.5} \\ 
\hline
\multicolumn{1}{|c|}{SNR} & \multicolumn{1}{c|}{1.812} & \multicolumn{1}{c|}{3.952} & \multicolumn{1}{c|}{5.807} & \multicolumn{1}{c|}{6.407} & \multicolumn{1}{c|}{5.464} & \multicolumn{1}{c|}{4.438} & \multicolumn{1}{c|}{3.773} & \multicolumn{1}{c|}{3.234} \\ 
\hline
\end{tabular}
\end{center}
\end{table}

\begin{table}[!htb]
\caption{Median values of ITER and SNR  over 10 random problem instances  with $5\leq r\leq 40$ and $\mbox{SNR}\in\{\infty, 20, 0\}$ for damped signals. The true model order is $r=20$.}\label{table:sensitivity_damped}
\begin{center}
\makegapedcells
\setcellgapes{3pt}
\small
\begin{tabular}{|ccccccccc}
\hline
\multicolumn{1}{|c|}{Test Rank} & \multicolumn{1}{c|}{$5$} & \multicolumn{1}{c|}{$10$} & \multicolumn{1}{c|}{$15$} & \multicolumn{1}{c|}{$20$} & \multicolumn{1}{c|}{$25$} & \multicolumn{1}{c|}{$30$} & \multicolumn{1}{c|}{$35$} & \multicolumn{1}{c|}{$40$} \\
\hline
\multicolumn{1}{|c|}{} & \multicolumn{8}{c|}{SNR$=\infty$} \\
\hline
\multicolumn{1}{|c|}{ITER} & \multicolumn{1}{c|}{43.5} & \multicolumn{1}{c|}{40.5} & \multicolumn{1}{c|}{48.5} & \multicolumn{1}{c|}{24}  & \multicolumn{1}{c|}{679.5} & \multicolumn{1}{c|}{942.5} & \multicolumn{1}{c|}{1014} & \multicolumn{1}{c|}{1130} \\
\hline 
\multicolumn{1}{|c|}{SNR} & \multicolumn{1}{c|}{2.224} & \multicolumn{1}{c|}{4.873} & \multicolumn{1}{c|}{9.000} & \multicolumn{1}{c|}{96.62}  & \multicolumn{1}{c|}{64.54} & \multicolumn{1}{c|}{61.20} & \multicolumn{1}{c|}{59.57} & \multicolumn{1}{c|}{59.00} \\ 
\hline
\multicolumn{1}{|c|}{} & \multicolumn{8}{c|}{SNR$=20$} \\
\hline
\multicolumn{1}{|c|}{ITER} & \multicolumn{1}{c|}{46} & \multicolumn{1}{c|}{40.5} & \multicolumn{1}{c|}{52.5} & \multicolumn{1}{c|}{26}  & \multicolumn{1}{c|}{3852} & \multicolumn{1}{c|}{4213} & \multicolumn{1}{c|}{6048} & \multicolumn{1}{c|}{5608} \\ 
\hline
\multicolumn{1}{|c|}{SNR} & \multicolumn{1}{c|}{2.223} & \multicolumn{1}{c|}{4.872} & \multicolumn{1}{c|}{8.999} & \multicolumn{1}{c|}{46.23}  & \multicolumn{1}{c|}{44.55} & \multicolumn{1}{c|}{43.36} & \multicolumn{1}{c|}{42.46} & \multicolumn{1}{c|}{41.65} \\ 
\hline
\multicolumn{1}{|c|}{} & \multicolumn{8}{c|}{SNR$=0$} \\
\hline
\multicolumn{1}{|c|}{ITER} & \multicolumn{1}{c|}{57.5} & \multicolumn{1}{c|}{74} & \multicolumn{1}{c|}{52.5} & \multicolumn{1}{c|}{36.5}  & \multicolumn{1}{c|}{2025} & \multicolumn{1}{c|}{1566} & \multicolumn{1}{c|}{2431} & \multicolumn{1}{c|}{3281} \\ 
\hline
\multicolumn{1}{|c|}{SNR} & \multicolumn{1}{c|}{2.217} & \multicolumn{1}{c|}{4.857} & \multicolumn{1}{c|}{8.904} & \multicolumn{1}{c|}{26.26}  & \multicolumn{1}{c|}{24.40} & \multicolumn{1}{c|}{23.18} & \multicolumn{1}{c|}{22.29} & \multicolumn{1}{c|}{21.52} \\ 
\hline
\end{tabular}
\end{center}
\end{table}

In practice, we may not know the exact model order  of a spectrally sparse signal but only have an estimation of it. Thus, it is of great interest to examine the performance of PGD when the model order is under- or over- estimated. The experiments are conducted for three-dimensional signals of the same size as in Section \ref{sec:nu_speed}. Here the true model order is $r=20$, and we observe $m=130\log(n)$ entries for  undamped signals while $m=0.03n$ entries for damped signals. The frequencies are generated randomly and the damping factors are generated in the same way as in Section \ref{sec:nu_speed}. Three noise levels are investigated: SNR$=\infty$ (noise-free), SNR$=20$ (light noise) and SNR$=0$ (heavy noise), and tests are conducted under the same additive noise model as in Section \ref{sec:robustness}. For a fixed noise level,   
we test PGD starting from $r=5$ and then increase the value of $r$ by $5$ each time until the maximum value  $40$  is reached. For each pair of $(\mbox{SNR},~r)$, $10$ random problem instances are tested, and PGD is terminated when $\|\bx^{k+1}-\bx^{k}\|_2/\|\bx^{k}\|_2\leq 10^{-5}$. The median values of ITER and SNR when convergence is attained are reported in Tables \ref{table:sensitivity_undamped} and \ref{table:sensitivity_damped} for  undamped and damped signals, respectively. 
As expected, PGD achieves the best SNR when the input value of $r$ is equal to $20$ (the true model order). The SNR of the estimation is usually very low when $r$ is smaller than $20$ due to  the systematic truncation error. On the other hand, even when $r$ is twice as large as the true model order, the SNR of the estimation is still desirable though it requires dramatically more number of iterations for PGD to converge. 

 \begin{figure}[!htb]
\centering
	\includegraphics[width=0.45 \textwidth]{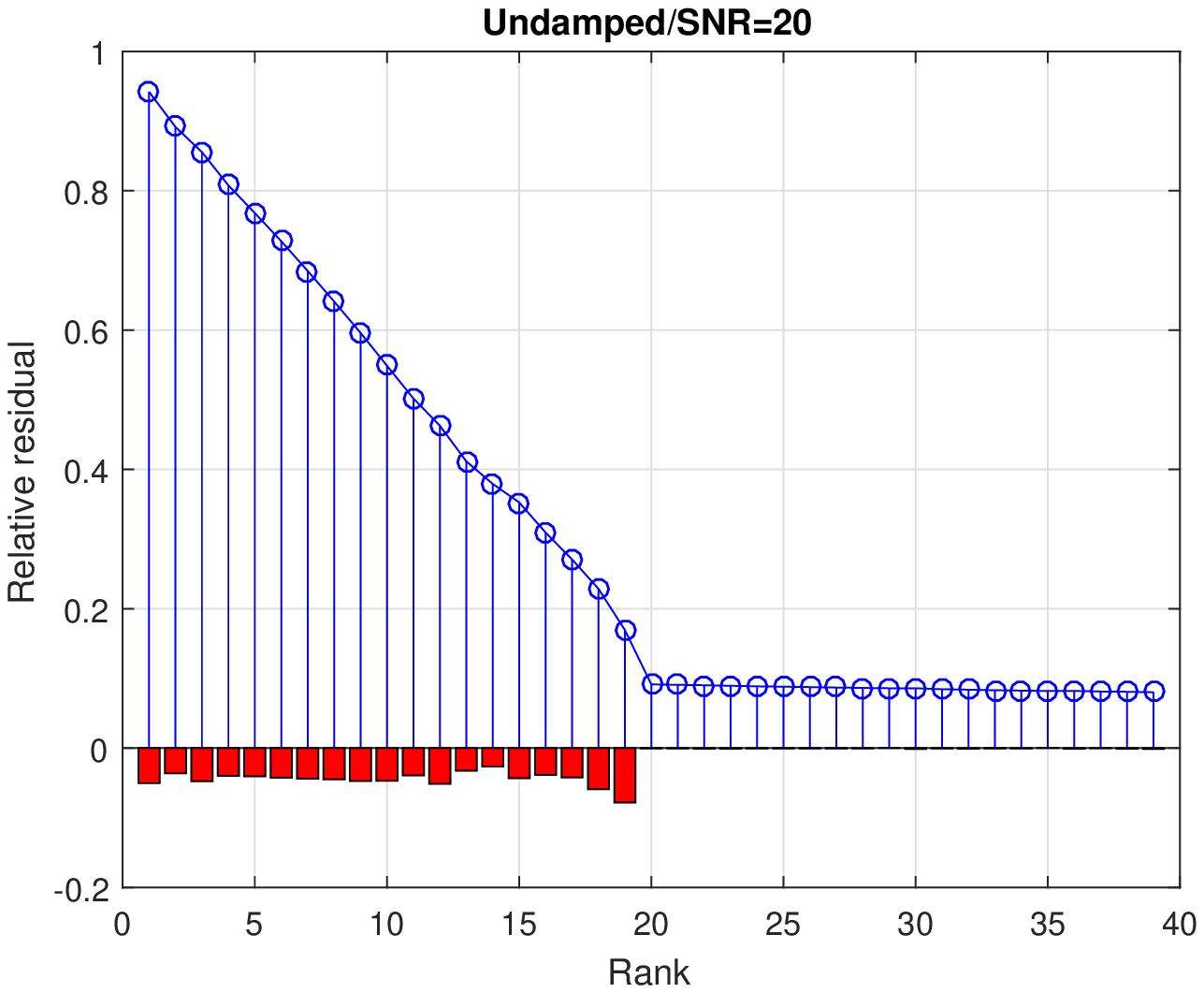}
	\includegraphics[width=0.45 \textwidth]{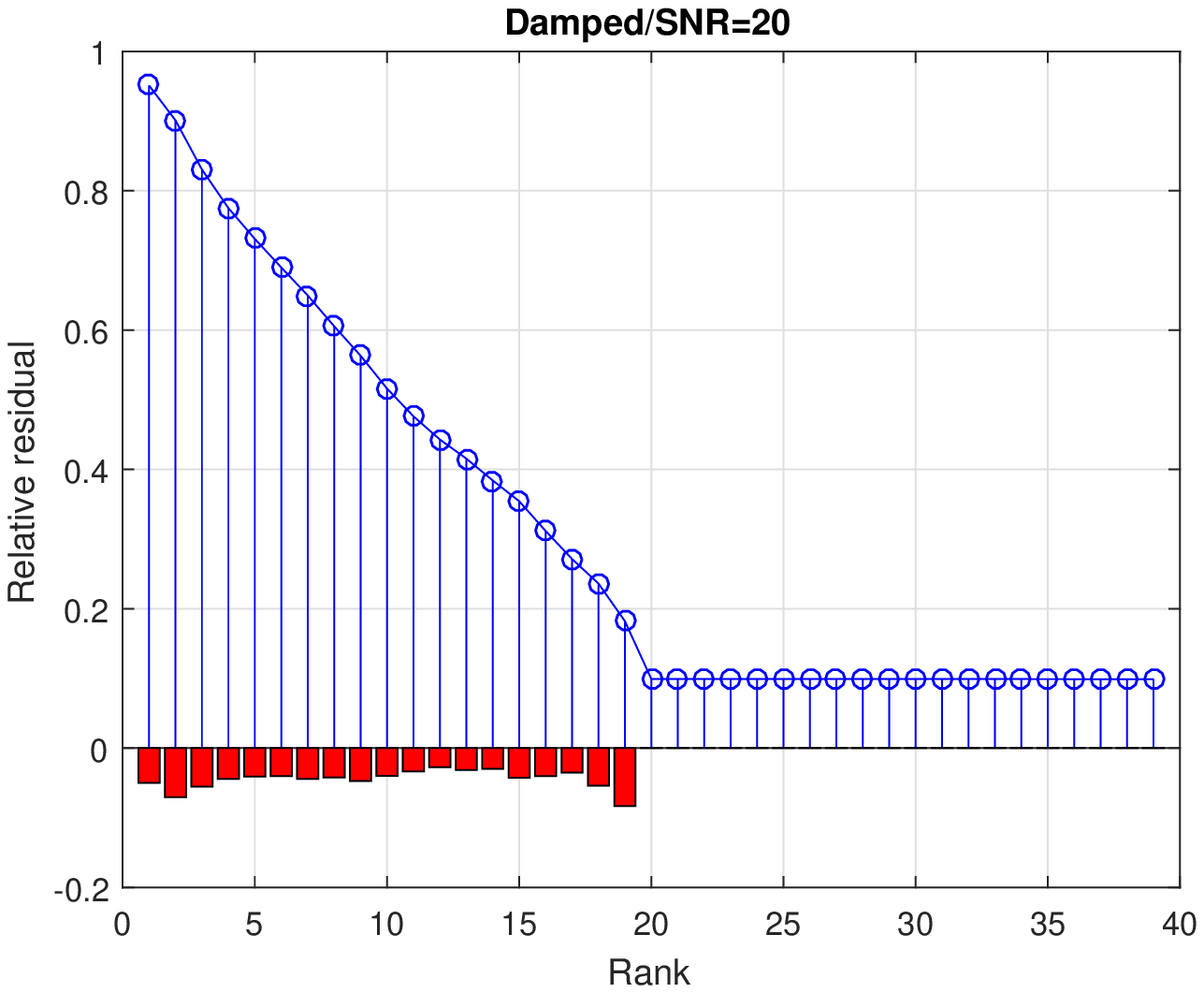}
\caption{ {Demonstration of rank increasing heuristic for  problem instances with $\mbox{SNR}=20$ for undamped (left) and damped (right) signals.}}
\label{fig:increasing-heuristics}
\end{figure}

Next, we suggest a rank increasing heuristic for PGD when the underlying model order is not known a priori. Starting from a sufficiently small $r$, we run PGD until convergence is reached (i.e., when $\|\bx^{k+1}-\bx^{k}\|_2/\|\bx^{k}\|_2\leq 10^{-5}$). Then we  compute and compare the relative residuals over the observed entries for the two successive  testing values of $r$. If the relative residual is improved significantly, we increase the value of $r$; otherwise the algorithm is terminated.  To validate the potential effectiveness of this heuristic, we test PGD for problem instances with SNR$=20$ for both undamped and damped signals, and with the values of $r$ increasing from 1 to 40. The computational results are presented in Figure~\ref{fig:increasing-heuristics}, where we show the relative residual plotted against the values of $r$, as well as the change of the relative residual when $r$ is increased by one. 
 The figure shows  that when $r$ is greater than  $20$, the improvement of the relative residuals  becomes very marginal for both undamped and damped signals.

%% file: proofs.tex
\section{Proof of Theorem~\ref{thm:exact_recovery}}\label{sec:pf}
The structure of the proof for Theorem~\ref{thm:exact_recovery} follows the typical two-step strategy in the convergence analysis of non-convex optimization algorithms: a {\em basin of attraction} is firstly
established, in which the algorithm  converges linearly to the true solution; and then it can be shown that the initial guess constructed in the algorithm lies inside the basin of attraction. We begin our presentation of the proof with a proposition about the initialization.
\begin{proposition}[Initialization Error]\label{prop:initial_error}
Suppose $\G\by$ is $\mu_0$-incoherent. If $m\geq c\hspace{0.05cm} \varepsilon_0^{-2}\mu c_s\kappa^2 r^2\log(n)$,  then one has $\BY\in\CS$ and 
\begin{align*}
\distsq{\BZ^0}{\BY}\leq 3\varepsilon^2_0\sigma_r(\G\by)\numberthis\label{eq:initial_error}
\end{align*}
with probability at least $1-n^{-2}$.
\end{proposition}
\begin{proof}
By  Lemma~\ref{lem:A1}, one has 
\begin{align*}
\ln\BL_0-\G\by\rn_2&\lesssim \sqrt{\frac{\mu_0c_sr\log(n)}{m}}\ln\G\by\rn_2\leq \sqrt{\frac{\mu c_sr\log(n)}{m}}\ln\G\by\rn_2\numberthis\label{eq:initial_eq1}
\end{align*}
with probability at least $1-n^{-2}$, where in the second inequality we use the assumption $\mu_0\leq \mu$. Together with the assumption on $m$, it follows immediately that 
\[
\sigma_1(\G\by)\leq \frac{\sigma_1(\BL^0)}{1-\varepsilon_0}.
\]
Consequently, one has $\BY\in\CS$ since 
$
\ln\BY\rn_{2,\infty}\leq \sqrt{\sigma_1(\G\by)}\max\{\ln\BU\rn_{2,\infty},\ln\BV\rn_{2,\infty}\}.
$
Moreover, one can easily see that $\BY\BR\in\CS$ for all $r$ by $r$ unitary matrices $\BR$.

Since $\BZ^0=\P_{\CS}(\TBZ^0)$ and $\BY\BR_{\TSZ^0}\in\CS$, one has 
\begin{align*}
\dist{\BZ^0}{\BY}\leq \ln\BZ^0-\BY\BR_{\TSZ^0}\rn_F\leq \ln\TBZ^0-\BY\BR_{\TSZ^0}\rn_F=\dist{\TBZ^0}{\BY}.\numberthis\label{eq:initial_eq2}
\end{align*}
Therefore, in order to show \eqref{eq:initial_error}, it suffices to bound $\dist{\TBZ^0}{\BY}$.
By Lemma~\ref{lem:A2}, one has
\begin{align*}
\distsq{\TBZ^0}{\BY}&\leq \frac{1}{2(\sqrt{2}-1)\sigma_r^2(\BY)}\ln\TBZ^0(\TBZ^0)^*-\BY\BY^*\rn_F^2\\
&=\frac{1}{4(\sqrt{2}-1)\sigma_r(\G\by)}\ln\TBZ^0(\TBZ^0)^*-\BY\BY^*\rn_F^2.\numberthis\label{eq:initial_eq3}
\end{align*}
Let $\BA$, $\BB$, $\BC$, and $\BD$ be four $s\times r$ complex matrices with $s\geq r$. 
A simple calculation yields
\begin{align*}
\la\BA\BA^*,\BB\BB^*\ra+\la\BC\BC^*,\BD\BD^*\ra&=\la\sum_{i=1}^r\ba_i\ba_i^*,\sum_{i=1}^r\bb\bb^*\ra+\la\sum_{i=1}^r\bc_i\bc_i^*,\sum_{i=1}^r\bd_i\bd_i^*\ra\\
&=\sum_{i,j=1}^r\lb\la\ba_i\ba_i^*,\bb_j\bb_j^*\ra+\la\bc_i\bc_i^*,\bd_j\bd_j^*\ra\rb\\
&=\sum_{i,j=1}^r\lb\la\ba_i^*\bb_j,\ba_i^*\bb_j\ra+\la\bc_i^*\bd_j,\bc_i^*\bd_j\ra\rb\\
&\geq 2\sum_{i,j=1}\Real\la \ba_i^*\bb_j,\bc_i^*\bd_j\ra\\
&=2\sum_{i,j=1}\Real\la \ba_i\bc_i^*,\bb_j\bd_j^*\ra\\
&=2\Real\la \BA\BC^*,\BB\BD^*\ra,\numberthis\label{eq:initial_eq4}
\end{align*}
where $\ba_i$, $\bb_i$, $\bc_i$ and $\bd_i$ are the $i$-th columns of $\BA$, $\BB$, $\BC$ and $\BD$ respectively. Then it follows that 
\begin{align*}
\ln\TBZ^0(\TBZ^0)^*-\BY\BY^*\rn_F^2 &=2\ln\BU^0\BS^0(\BV^0)^*-\BU\BS\BV^*\rn_F^2+\ln\BU^0\BS^0(\BU^0)^*-\BU\BS\BU^*\rn_F^2\\
&\quad+\ln\BV^0\BS(\BV^0)^*-\BV\BS\BV^*\rn_F^2\\
&\leq 4\ln\BU^0\BS^0(\BV^0)^*-\BU\BS\BV^*\rn_F^2=4\ln\BL^0-\G\by\rn_F^2\numberthis\label{eq:initial_eq5}
\end{align*}
where the inequality follows from 
\begin{align*}
&\ln\BU^0\BS^0(\BU^0)^*-\BU\BS\BU^*\rn_F^2+\ln\BV^0\BS^0(\BV^0)^*-\BV\BS\BV^*\rn_F^2\leq 2\ln\BU^0\BS^0(\BV^0)^*-\BU\BS\BV^*\rn_F^2,
\end{align*}
which can be easily verified using \eqref{eq:initial_eq4}. Substituting \eqref{eq:initial_eq5} into \eqref{eq:initial_eq3}
gives 
\begin{align*}
\distsq{\TBZ^0}{\BY}\leq\frac{1}{(\sqrt{2}-1)\sigma_r(\G\by)}\ln\BL^0-\G\by\rn_F^2.
\end{align*}
Since 
\begin{align*}
\ln\BL^0-\G\by\rn_F\lesssim\sqrt{\frac{\mu c_sr^2\log(n)}{m}}\ln\G\by\rn_2\leq \varepsilon_0\sigma_r(\G\by),
\end{align*}
we finally have 
\begin{align*}
\distsq{\BZ^0}{\BY}\leq\distsq{\TBZ^0}{\BY}\leq 3\varepsilon^2_0\sigma_r(\G\by),
\end{align*}
which completes the proof of \eqref{eq:initial_error}.
\end{proof}
With Proposition~\ref{prop:initial_error} in place, the proof of Theorem~\ref{thm:exact_recovery} is complete if we can establish the local contraction property of Algorithm~\ref{alg:pgd}, as stated in the following proposition.
\begin{proposition}[Local Contraction]\label{prop:local_conv}
Assume $\BY\in\CS$. Let $\varepsilon_0$ be an absolute constant obeying $0<\varepsilon_0\leq\frac{1}{11}$. For any matrix \kw{$\BZ\in\CS$}, define 
\begin{align*}
\TBZ=\BZ-\eta\nabla F(\BZ)\quad\mbox{and}\quad\BZ^+=\PC{\TBZ}.
\end{align*}
There exists a numerical constant $\nu=\frac{1}{10}\sigma_r(\G\by)$ such that with probability at least $1-c_1\cdot n^{-2}$, 
\begin{align*}
\distsq{\BZ^{+}}{\BY}\leq (1-\eta\nu) \distsq{\BZ}{\BY}
\end{align*}
holds for all $\BZ$ obeying \kw{$\distsq{\BZ}{\BY}\leq 3\varepsilon_0^2\sigma_r(\G\by)$} provided 
\kw{\begin{align*}
m\geq c_2\hspace{0.05cm} \varepsilon_0^{-2}\mu^2c_s^2\kappa^2r^2\log(n)\quad\mbox{and}\quad\eta\leq\frac{\sigma_r(\G\by)}{600(\mu c_sr)^2\sigma_1^2(\G\by)}
\end{align*}}
\end{proposition}
Based on the same argument as in \eqref{eq:initial_eq2}, one has $\dist{\BZ^+}{\BY}\leq\dist{\TBZ}{\BY}$. Hence, it suffices to show   that
\begin{align*}
\distsq{\TBZ}{\BY}\leq (1-\eta\nu) \distsq{\BZ}{\BY}\numberthis\label{eq:contraction_obj}
\end{align*}
holds for all matrices $\BZ$ within a small neighborhood of $\BY$. 
Let $\BH=\BZ-\BY\BR_{\SZ}$. We follow a similar route as in \cite{FGD} and instead establish the {\em regularity condition}
\begin{align*}
\Real\la\nabla F(\BZ),\BH\ra\geq \frac{\eta}{2}\ln \nabla F(\BZ) \rn_F^2+\frac{\nu}{2}\ln\BH\rn_F^2\numberthis \label{eq:regularity_cond}
\end{align*}
for all matrices $\BZ$ that are sufficiently close to $\BY$.
The notation of regularity condition was first introduced in \cite{candes2015phase} to show the convergence of a non-convex gradient descent algorithm for phase retrieval and since then has been extended to many other problems, see \cite{FGD} and references therein. Once \eqref{eq:regularity_cond} is established, a little algebra yields
\begin{align*}
\distsq{\TBZ}{\BY}&=\ln\TBZ-\BY\BR_{\TSZ}\rn_F^2\leq \ln\TBZ-\BY\BR_{\SZ}\rn_F^2\\
&=\ln\BH\rn_F^2+\eta^2\ln\nabla F(\BZ)\rn_F^2-2\eta\Real\la\nabla F(\BZ),\BH\ra\\
&\leq (1-\eta\nu)\ln\BH\rn_F^2\\
&=(1-\eta\nu)\distsq{\BZ}{\BY}.
\end{align*}

The proof of the regularity condition will occupy the remainder of this section. Even though the proof follows a well-established route, especially that in \cite{FGD},  the details of the proof are nevertheless quite involved and technical. Firstly, our objective function involves a transformation from the matrix domain to the vector domain, and an extra regularizer is also included to preserve the Hankel structure of the matrix. Secondly,  we need to establish  a key lemma which is closely related to the second largest eigenvalue of a special random graph, as presented in the next subsection.
\subsection{A Key Ingredient}
The following lemma will play a key role in the proof of the regularity condition.
\begin{lemma}\label{lem:key}
Suppose $\Omega=\{a_k\}_{k=1}^m$, where each $a_k$ is sampled from $\{0,\cdots,n-1\}$ independently and uniformly with replacement. Then for all $\bz\in\R^{n_1}$ and $\bw\in\R^{n_2}$,
\begin{align*}
p^{-1}\sum_{k=1}^{m}\sum_{i+j=a_k}z_iw_j\leq\ln\bz\rn_1\ln\bw\rn_1+\sqrt{\frac{24n\log(n)}{p}}\ln\bz\rn_2\ln\bw\rn_2
\end{align*}
holds with probability at least $1-2n^{-2}$ provided  $m\geq \frac{8}{3}\log(n)$.
\end{lemma}
\begin{proof}
Let $\BHH_a,~a=0,\cdots,n-1$, be  an $n_1\times n_2$ matrix with the $a$-th skew-diagonal entries being equal to one and all the other entries being equal to zero.
Notice that $p^{-1}\sum_{k=1}^m\sum_{i+j=a_k}z_iw_j$ can be written as 
\begin{align*}
p^{-1}\sum_{k=1}^m\sum_{i+j=a_k}z_iw_j &= p^{-1}\sum_{k=1}^m\bz^T\bm{H}_{a_k}\bw=\bz^T\lb\frac{n}{m}\sum_{k=1}^m\bm{H}_{a_k}\rb\bw\\
&=\bz^T\lb\bone_{n_1}\bone_{n_2}^T\rb\bw+\bz^T\lb\frac{n}{m}\sum_{k=1}^m\bm{H}_{a_k}-\bone_{n_1}\bone_{n_2}^T\rb\bw\\
&\leq \ln\bz\rn_1\ln\bw\rn_1+\ln \sum_{k=1}^m\lb \frac{n}{m}\bm{H}_{a_k}-\frac{1}{m}\bone_{n_1}\bone_{n_2}^T\rb\rn_2\ln\bz\rn_2\ln\bw\rn_2.
\numberthis\label{eq:key_eq1}
\end{align*}
Let $\BZ_k= \frac{n}{m}\bm{H}_{a_k}-\frac{1}{m}\bone_{n_1}\bone_{n_2}^T$. One can easily see that $\mean{\BZ_k}=0$ and 
\begin{align*}
\ln\BZ_k\rn_2\leq \ln\frac{n}{m}\bm{H}_{a_k}\rn_2+\ln\frac{1}{m}\bone_{n_1}\bone_{n_2}^T\rn_2\leq\frac{2n}{m}.
\end{align*}
Moreover, one has 
\begin{align*}
\mean{\BZ_k\BZ_k^T} &= \mean{\lb\frac{n}{m}\bm{H}_{a_k}-\frac{1}{m}\bone_{n_1}\bone_{n_2}^T\rb\lb\frac{n}{m}\bm{H}_{a_k}^T-\frac{1}{m}\bone_{n_2}\bone_{n_1}^T\rb}\\
&=\frac{n^2}{m^2}\mean{\bm{H}_{a_k}\bm{H}_{a_k}^T}-\frac{n}{m^2}\lb\bone_{n_1}\bone_{n_2}^T\rb\mean{\bm{H}_{a_k}}
-\frac{n}{m^2}\mean{\bm{H}_{a_k}}\lb\bone_{n_2}\bone{n_1}^T\rb+\frac{n_2}{m^2}\bone_{n_1}\bone_{n_1}^T\\
&=\frac{n}{m^2}\sum_{a=1}^{n-1}\bm{H}_a\bm{H}_a^T-\frac{n_2}{m^2}\bone_{n_1}\bone_{n_1}^T\\
&=\frac{n_2}{m^2}\lb n\BI_{n_1}-\bone_{n_1}\bone_{n_1}^T\rb,
\end{align*}
so $\ln\mean{\BZ_k\BZ_k^T}\rn_2\leq \frac{2n^2}{m^2}$. Similarly, one also has $\ln\mean{\BZ_k^T\BZ_k}\rn_2\leq \frac{2n^2}{m^2}$. Consequently, 
\begin{align*}
\max\lcb\ln\sum_{k=1}^m\mean{\BZ_k\BZ_k^T}\rn_2,\ln\sum_{k=1}^m\mean{\BZ_k^T\BZ_k}\rn_2\rcb\leq \frac{2n^2}{m}.
\end{align*}
Thus, the application of the Bernstein's inequality (see for example \cite[Theorem~1.6]{Tropp}) yields
\begin{align*}
\prob{\ln\sum_{k=1}^m\BZ_k\rn_2>t}\leq(n_1+n_2)\exp\lb\frac{-t^2/2}{2n^2/m+2nt/3m}\rb.
\end{align*}
Letting $t=\sqrt{\frac{24n^2\log(n)}{m}}$ gives 
\begin{align*}
\prob{\ln\sum_{k=1}^m\BZ_k\rn_2>t}\leq 2n^{-2}
\end{align*}
provided $m\geq\frac{8}{3}\log(n)$. Substituting this result into \eqref{eq:key_eq1} concludes the proof. 
\end{proof}
\begin{remark*}{\normalfont
Suppose $n$ is odd and $n_1=n_2=(n+1)/2$. Let $\BH$ be an $n_1\times n_1$ random Hankel matrix, each skew-diagonal of which takes the value $1$ with probability $p$ and the value $0$ with probability $1-p$. Then  $\BH$ can be viewed as the adjacency matrix corresponding  a special random graph. Without rigorous justification, we can see that the largest eigenvalue of $\BH$, denoted $\lambda_1$, is   of order about $n_1p$ as $\mean{\BH}=p\bone_{n_1}\bone_{n_1}^T$.  Let $\lambda_2$ be the second largest (in magnitude) eigenvalue of $\BH$. Roughly speaking, Lemma~\ref{lem:key}  says that $|\lambda_2|\approx \sqrt{n_1p\log(n_1)}$ since $|\lambda_2|$ can be approximated by $\ln\BH-\bone_{n_1}\bone_{n_1}^T\rn_2$. Let $\BG$ be an $n_1\times n_1$ adjacency matrix of a random graph with $n_1$ vertex and  every edge of which is connected with probability $p$. That is, each entry of $\BG$ takes the value  $1$ with probability $p$ and the value $0$ with probability $1-p$. It was shown in \cite{Geig_kw} the second largest (in magnitude) eigenvalue of $\BG$ is of order at most $\sqrt{n_1p}$, which has also been extended to singular values in \cite{keshavan2010matrix}. Thus,  our analysis looses a $\sqrt{\log(n_1)}$ factor compared to the result for $\BG$. However, we want to emphasize that the extra $\sqrt{\log(n)}$ factor in Lemma~\eqref{lem:key} does not affect our final result as a log factor will also appear in other place. That being said, we conjecture that the extra $\sqrt{\log(n)}$ factor for $\BH$ is just an artifact of our analysis framework which uses the Bernstein's inequality under the sampling with replacement model, and it can be eliminated  by the spectral techniques used in \cite{Geig_kw} under the Bernoulli model. We leave this for future work.}\end{remark*}
\subsection{Proof of the Regularity Condition}
The goal of this subsection is  to show that the regularity condition \eqref{eq:regularity_cond} holds with high probability. Before proceeding to the formal proof, we first 
consider the expectation of ${\Real\la\nabla F(\BZ),\BH\ra}$ and see what lower bound can be anticipated.
With a slight abuse of notation, we denote $\BY\BR_{\SZ}$ by $\BY$ throughout this subsection for ease of presentation. Since there exists a close solution for $\BR_{\SZ}$, as presented in \eqref{eq:R_form}, one can easily verify that 
\begin{align*}\BH^*\BY=\BY^*\BH\quad\mbox{and}\quad\BY^*\BZ=\BZ^*\BY\succeq 0.\numberthis\label{eq:H_prop}
\end{align*}

By noting that $\mean{p^{-1}\P_\Omega}=\I$, the expectation of $\mean{\Real\la\nabla f(\BZ),\BH\ra}$ can be  bounded  below as
\begin{align*}
&\mean{\Real\la\nabla f(\BZ),\BH\ra}\\ &= \Real\la\BZ_\SU\BZ_\SV^*-\BY_\SU\BY_\SV^*,\BH_\SU\BZ_\SV^*+\BZ_\SU\BH_\SV^*\ra\\
&=\Real\la \BY_{\SU}\BH_{\SV}^*+\BH_{\SU}\BY_{\SV}^*+\BH_{\SU}\BH_{\SV}^*,\BY_{\SU}\BH_{\SV}^*+\BH_{\SU}\BY_{\SV}^*+2\BH_{\SU}\BH_{\SV}^*\ra\\
&=\ln \BY_{\SU}\BH_{\SV}^*+\BH_{\SU}\BY_{\SV}^*\rn_F^2+3\Real\la \BY_{\SU}\BH_{\SV}^*+\BH_{\SU}\BY_{\SV}^*, \BH_\SU\BH_\SV^*\ra+2\ln\BH_\SU\BH_\SV^*\rn_F^2\\
&\geq \ln \BY_{\SU}\BH_{\SV}^*+\BH_{\SU}\BY_{\SV}^*\rn_F^2-3\ln \BY_{\SU}\BH_{\SV}^*+\BH_{\SU}\BY_{\SV}^*\rn\ln\BH_\SU\BH_\SV^*\rn+2\ln\BH_\SU\BH_\SV^*\rn_F^2\\
&\geq \frac{1}{2}\ln \BY_{\SU}\BH_{\SV}^*+\BH_{\SU}\BY_{\SV}^*\rn_F^2-\frac{5}{2}\ln\BH_\SU\BH_\SV^*\rn_F^2\\
&=\frac{1}{2}\lb\ln\BY_\SU\BH_\SV^*\rn_F^2+\ln \BH_{\SU}\BY_{\SV}^*\rn_F^2\rb-\frac{5}{2}\ln\BH_\SU\BH_\SV^*\rn_F^2+\Real\la\BH_{\SU}^*\BY_\SU, \BY_{\SV}^*\BH_\SV\ra,\numberthis\label{eq:mean_f}
\end{align*}
where in the second line we use $\BZ=\BY+\BH$, and in the third line we use the inequality $a^2-3ab+2b^2\geq \frac{1}{2}a^2-\frac{5}{2}b^2$. 

Before continuing to bound $\mean{\Real\la\nabla F(\BZ),\BH\ra}$ by adding $\lambda\Real\la\nabla g(\BZ),\BH\ra$ to $\mean{\Real\la\nabla f(\BZ),\BH\ra}$, it might be better to  examine the role of $g(\BZ)$ by studying a special case.
Suppose $\BH_\SU=\delta\cdot\BY_\SU$ and $\BH_\SV=-\delta\cdot\BY_\SV$, where $\delta>0$ is a small numerical constant. 
Then one has
\begin{align*}
\mean{\Real\la\nabla f(\BZ),\BH\ra}=2\ln\BH_\SU\BH_\SV^*\rn_F^2=2\delta^4\ln\BY_\SU\BY_\SV^*\rn_F^2=2\delta^4\ln\BS\rn_F^2,\end{align*}
where the last equality follows from the fact $\BY_\SU\BY_\SV^*=\BU\BS\BV^*$.
Since $\ln\BH\rn_F^2=\delta^2\ln\BY\rn_F^2=2\delta^2\ln\BS\rn_*$, the regularity condition \eqref{eq:regularity_cond} cannot be true for $f(\BZ)$ without the regularization function $g(\BZ)$. In this case, one can observe that the mismatch between $\BZ_{\SU}^*\BZ_{\SU}$ and  $\BZ_{\SV}^*\BZ_{\SV}$ increases compared with the mismatch between $\BY_{\SU}^*\BY_{\SU}$ and  $\BY_{\SV}^*\BY_{\SV}$ which is equal to zero. Because $g(\BZ)$ penalizes the mismatch between  $\BZ_{\SU}^*\BZ_{\SU}$ and  $\BZ_{\SV}^*\BZ_{\SV}$, one may intuitively expect that it can control the occurrence of this case so that $F(\BZ)=f(\BZ)+\lambda g(\BZ)$ could obey the regularity condition.

Let $\BD = \begin{bmatrix}\BI_{n_1}&\bzero\\\bzero&-\BI_{n_2}\end{bmatrix}$. 
We can  bound $\Real\la\nabla g(\BZ),\BH\ra$ from below as 
\begin{align*}
&\Real\la\nabla g(\BZ),\BH\ra\\  &= \Real\la\BD\BZ(\BZ^*\BD\BZ),\BH\ra=\Real\la\BZ^*\BD\BZ,\BZ^*\BD\BH\ra\\
&=\Real\la \BY^*\BD\BH+\BH^*\BD\BY+\BH^*\BD\BH, \BY^*\BD\BH+\BH^*\BD\BH\ra\\
&=\ln \BY^*\BD\BH\rn_F^2+3\Real\la\BY^*\BD\BH,\BH^*\BD\BH\ra+\ln\BH^*\BD\BH\rn_F^2+\Real\la \BY^*\BD\BH,\BH^*\BD\BY\ra\\
&=\frac{1}{2}\ln \BY^*\BD\BH\rn_F^2+\frac{1}{2}\ln \BY^*\BD\BH+3\BH^*\BD\BH\rn_F^2-\frac{7}{2}\ln\BH^*\BD\BH\rn_F^2\\
&\quad+\Real\la \BY^*\BD\BH,\BH^*\BD\BY\ra\\
&=\frac{1}{2}\ln \BY^*\BD\BH\rn_F^2+\frac{1}{2}\ln \BY^*\BD\BH+3\BH^*\BD\BH\rn_F^2-\frac{7}{2}\ln\BH^*\BD\BH\rn_F^2\\
&\quad+\Real\la\BY^*\BH,\BH^*\BY\ra-4\Real\la\BH_{\SU}^*\BY_\SU, \BY_{\SV}^*\BH_\SV\ra\\
&\geq \frac{1}{2}\ln \BY^*\BD\BH\rn_F^2-\frac{7}{2}\ln\BH^*\BD\BH\rn_F^2\\
&\quad-4\Real\la\BH_{\SU}^*\BY_\SU, \BY_{\SV}^*\BH_\SV\ra,\numberthis\label{eq:mean_g}
\end{align*}
where the third equality  follows from $\BY^*\BD\BY=\bzero$, the fourth equality follows from
\begin{align*}
\Real\la \BH^*\BD\BY,\BH^*\BD\BH\ra=\Real\la \BY^*\BD\BH,\BH^*\BD\BH\ra=\Real\la \BH^*\BD\BH,\BY^*\BD\BH\ra,
\end{align*}
 the last equality follows from 
\begin{align*}
\Real\la\BH_{\SU}^*\BY_\SU, \BY_{\SV}^*\BH_\SV\ra = \Real\la\BY_{\SV}^*\BH_\SV,\BH_{\SU}^*\BY_\SU\ra=\Real\la\BY_\SU^*\BH_\SU,\BH_\SV^*\BY_\SV\ra,
\end{align*}
and the inequality follows from $\BH^*\BY=\BY^*\BH$, see \eqref{eq:H_prop}.

If we take $\lambda=\frac{1}{4}$, then combining \eqref{eq:mean_f} and \eqref{eq:mean_g} together implies 
\begin{align*}
\mean{\Real\la\nabla F(\BZ),\BH\ra}&\geq \frac{1}{2}\lb\ln\BY_\SU\BH_\SV^*\rn_F^2+\ln \BH_{\SU}\BY_{\SV}^*\rn_F^2\rb-\frac{5}{2}\ln\BH_\SU\BH_\SV^*\rn_F^2-\frac{7}{8}\ln\BH^*\BD\BH\rn_F^2\\
&\quad+\frac{1}{8}\ln\BY^*\BD\BH\rn_F^2\\
&\gtrsim\lb \sigma_r(\G\by)-\ln\BH\rn_F^2\rb\ln\BH\rn_F^2+\ln\BY^*\BD\BH\rn_F^2.\numberthis\label{eq:mean_lower}
\end{align*}
That is, we have established a lower bound for the expectation of $\Real\la\nabla F(\BZ),\BH\ra$. As we will show later, $\Real\la\nabla F(\BZ),\BH\ra$ obeys a similar lower bound with high probability. Moreover, the right hand side of \eqref{eq:regularity_cond} can be bounded from above by a similar bound. Therefore, $F(\BZ)$ obeys the regularity condition for sufficiently small $\BH$. Specifically, we are going to show  the following two bounds,
\begin{align*}
&\Real\la\nabla F(\BZ),\BH\ra\geq \frac{1}{10}\sigma_r(\G\by)\ln\BH\rn_F^2+\frac{1}{8}\ln\BY^*\BD\BH\rn_F^2\numberthis\label{eq:local_cur},\\
&\ln\nabla F(\BZ)\rn_F^2\leq60(\mu c_s r)^2 \sigma_1^2(\G\by)\ln\BH\rn_F^2+\frac{1}{2}\sigma_1(\G\by)\ln\BY^*\BD\BH\rn_F^2,\numberthis\label{eq:local_smooth}
\end{align*}
hold with high probability provided \kw{$\ln\BH\rn_F^2\leq 3\varepsilon_0^2\sigma_r(\G\by)$ and $m\gtrsim\varepsilon_0^{-2}\mu^2 c_s^2\kappa^2r^2\log(n)$} for $\varepsilon_0\leq \frac{1}{11}$. The above two inequalities are typically referred to as the {\em local curvature property} and the {\em local smooth property} of the function $F(\BZ)$ in the literature, see for example \cite{candes2015phase,FGD}. Once they are established,  one can easily see that $F(\BZ)$ obeys the regularity condition \eqref{eq:regularity_cond} with $$\kw{\eta\leq\frac{\sigma_r(\G\by)}{600(\mu c_sr)^2\sigma_1^2(\G\by)}\quad\mbox{and}\quad\nu=\frac{1}{10}\sigma_r(\G\by)}.$$
\subsubsection{Proof of \eqref{eq:local_cur}}
Since $\Real\la \nabla g(\BZ),\BH\ra$ is deterministic and we have already obtained its lower bound in \eqref{eq:mean_g}, it only remains  to  work out the lower bound for $\Real\la \nabla f(\BZ),\BH\ra$ and then combine it together with that for $\Real\la \nabla g(\BZ),\BH\ra$. Note that 
\begin{align*}
&\Real\la \nabla f(\BZ),\BH\ra \\& = \Real\la (\I-\G\G^*)(\BZ_\SU\BZ_\SV^*)+p^{-1}\G\P_\Omega\G^*(\BZ_\SU\BZ_\SV^*-\BY_\SU\BY_\SV^*),\BH_\SU\BZ_\SV^*+\BZ_\SU\BH_\SV^*\ra\\
&=\Real\la (\I-\G\G^*)(\BZ_\SU\BZ_\SV^*-\BY_\SU\BY_\SV^*)+p^{-1}\G\P_\Omega\G^*(\BZ_\SU\BZ_\SV^*-\BY_\SU\BY_\SV^*),\BH_\SU\BZ_\SV^*+\BZ_\SU\BH_\SV^*\ra\\
&=\Real\la(\I-\G\G^*)(\BH_\SU\BY_\SV^*+\BY_\SU\BH_\SV^*+\BH_\SU\BH_\SV^*),\BH_\SU\BY_\SV^*+\BY_\SU\BH_\SV^*+2\BH_\SU\BH_\SV^*\ra\\
&\quad+\Real\la p^{-1}\G\P_\Omega\G^*(\BH_\SU\BY_\SV^*+\BY_\SU\BH_\SV^*+\BH_\SU\BH_\SV^*),\BH_\SU\BY_\SV^*+\BY_\SU\BH_\SV^*+2\BH_\SU\BH_\SV^*\ra\\
&:=I_1+I_2\numberthis\label{eq:grad_f},
\end{align*}
where the second equality follows from the fact $(\I-\G\G^*)(\BY_\SU\BY_\SV^*)=\bzero$. 

\textbf{Lower bound for $I_1$.}  The first term $I_1$ can be bounded directly as follows:
\begin{align*}
I_1 &= \Real\la(\I-\G\G^*)(\BH_\SU\BY_\SV^*+\BY_\SU\BH_\SV^*+\BH_\SU\BH_\SV^*),(\I-\G\G^*)(\BH_\SU\BY_\SV^*+\BY_\SU\BH_\SV^*+2\BH_\SU\BH_\SV^*)\ra\\
&\geq \ln (\I-\G\G^*)(\BH_\SU\BY_\SV^*+\BY_\SU\BH_\SV^*)\rn_F^2\\
&\quad-3\ln (\I-\G\G^*)(\BH_\SU\BY_\SV^*+\BY_\SU\BH_\SV^*)\rn_F\ln (\I-\G\G^*)(\BH_\SU\BH_\SV^*)\rn_F+2\ln (\I-\G\G^*)(\BH_\SU\BH_\SV^*)\rn_F^2\\
&\geq \frac{11}{20}\ln (\I-\G\G^*)(\BH_\SU\BY_\SV^*+\BY_\SU\BH_\SV^*)\rn_F^2-3\ln (\I-\G\G^*)(\BH_\SU\BH_\SV^*)\rn_F^2,
\end{align*}
where the first equality follows from that $\G\G^*$ is a projection operator, and the second inequality follows from $a^2-3ab+2b^2\geq \frac{11}{20}a^2-3b^2$.

\textbf{Lower bound for $I_2$.} Recall from Section~\ref{sec:alg:hankel} that $w_a$, $a=0,\cdots,n-1$, denotes the number of entries in the skew-diagonal of an $n_1\times n_2$ matrix. Let $\BG_a,~a=0,\cdots,n-1$, be  an $n_1\times n_2$ matrix with the $a$-th skew-diagonal entries being equal to $1/\sqrt{w_a}$ and all the other entries being equal to zero. Then,
\begin{align*}
\G^*(\BH_\SU\BY_\SV^*+\BY_\SU\BH_\SV^*) = \lcb\la\BG_{a}, \BH_\SU\BY_\SV^*+\BY_\SU\BH_\SV^*\ra\rcb_{a=0}^{n-1}
\end{align*}
and 
\begin{align*}
\G^*(\BH_\SU\BH_\SV^*) = \lcb\la\BG_a,\BH_\SU\BH_\SV^*\ra\rcb_{a=0}^{n-1}.
\end{align*}
It follows that 
\begin{align*}
&\Real\la \G\P_\Omega\G^*(\BH_\SU\BY_\SV^*+\BY_\SU\BH_\SV^*+\BH_\SU\BH_\SV^*),\BH_\SU\BY_\SV^*+\BY_\SU\BH_\SV^*+2\BH_\SU\BH_\SV^*\ra\\
&=\Real\la\P_\Omega\G^*(\BH_\SU\BY_\SV^*+\BY_\SU\BH_\SV^*+\BH_\SU\BH_\SV^*),\G^*(\BH_\SU\BY_\SV^*+\BY_\SU\BH_\SV^*+2\BH_\SU\BH_\SV^*)\ra\\
&=\la \P_\Omega\G^*(\BH_\SU\BY_\SV^*+\BY_\SU\BH_\SV^*),\G^*(\BH_\SU\BY_\SV^*+\BY_\SU\BH_\SV^*) \ra\\
&\quad+3\Real\la \P_\Omega\G^*(\BH_\SU\BY_\SV^*+\BY_\SU\BH_\SV^*),\G^*(\BH_\SU\BH_\SV^*)\ra+2\la \P_\Omega\G^*(\BH_\SU\BH_\SV^*),\G^*(\BH_\SU\BH_\SV^*)\ra\\
&=\sum_{k=1}^m\lab\la \BG_{a_k},\BH_\SU\BY_\SV^*+\BY_\SU\BH_\SV^*\ra\rab^2+3\Real\lb\sum_{k=1}^m\overline{\la \BG_{a_k},\BH_\SU\BY_\SV^*+\BY_\SU\BH_\SV^*\ra}{\la \BG_{a_k},\BH_\SU\BH_\SV^*\ra}\rb\\
&\quad+2\sum_{k=1}^m\lab\la \BG_{a_k},\BH_\SU\BH_\SV^*\ra\rab^2\\
&\geq \sum_{k=1}^m\lab\la \BG_{a_k},\BH_\SU\BY_\SV^*+\BY_\SU\BH_\SV^*\ra\rab^2-3\sqrt{\sum_{k=1}^m\lab\la \BG_{a_k},\BH_\SU\BY_\SV^*+\BY_\SU\BH_\SV^*\ra\rab^2}\sqrt{\sum_{k=1}^m\lab\la \BG_{a_k},\BH_\SU\BH_\SV^*\ra\rab^2}\\
&\quad+2\sum_{k=1}^m\lab\la \BG_{a_k},\BH_\SU\BH_\SV^*\ra\rab^2\\
&\geq \frac{11}{20} \sum_{k=1}^m\lab\la \BG_{a_k},\BH_\SU\BY_\SV^*+\BY_\SU\BH_\SV^*\ra\rab^2-3\sum_{k=1}^m\lab\la \BG_{a_k},\BH_\SU\BH_\SV^*\ra\rab^2\\
&=\frac{11}{20}\la \P_\Omega\G^*(\BH_\SU\BY_\SV^*+\BY_\SU\BH_\SV^*),\G^*(\BH_\SU\BY_\SV^*+\BY_\SU\BH_\SV^*) \ra-3\sum_{k=1}^m\lab\la \BG_{a_k},\BH_\SU\BH_\SV^*\ra\rab^2,
\end{align*}
where the third equality and the last equality follow from \eqref{eq:tmp112}, the first inequality follows from the H\"older inequality, and the second inequality follows from $a^2-3ab+2b^2\geq \frac{11}{20}a^2-3b^2$. Consequently,
\begin{align*}
I_2&\geq \frac{11}{20}\la p^{-1}\P_\Omega\G^*(\BH_\SU\BY_\SV^*+\BY_\SU\BH_\SV^*),\G^*(\BH_\SU\BY_\SV^*+\BY_\SU\BH_\SV^*) \ra\\
&\quad-3p^{-1}\sum_{k=1}^m\lab\la \BG_{a_k},\BH_\SU\BH_\SV^*\ra\rab^2.\numberthis\label{eq:grad_f_i2}
\end{align*} 
We can bound $p^{-1}\sum_{k=1}^m\lab\la \BG_{a_k},\BH_\SU\BH_\SV^*\ra\rab^2$ from above by  Lemma~\ref{lem:key} as follows:
\begin{align*}
&p^{-1}\sum_{k=1}^m\lab\la \BG_{a_k},\BH_\SU\BH_\SV^*\ra\rab^2\\&=p^{-1}\sum_{k=1}^m\lab\frac{1}{\sqrt{w_{a_k}}}\sum_{i+j=a_k}\la \be_i\be_j^T,\BH_\SU\BH_\SV^*\ra\rab^2&\\
&\leq p^{-1}\sum_{k=1}^m\sum_{i+j=a_k}\lab\la \be_i\be_j^T,\BH_\SU\BH_\SV^*\ra\rab^2\\
&\leq p^{-1}\sum_{k=1}^m\sum_{i+j=a_k}\ln\BH_\SU^{(i,:)}\rn_2^2\ln\BH_\SV^{(j,:)}\rn_2^2\\
&\leq \ln\BH_\SU\rn_F^2\ln\BH_\SV\rn_F^2+\sqrt{\frac{24n\log(n)}{p}}\sqrt{\sum_{i=1}^{n_1}\ln\BH_\SU^{(i,:)}\rn_2^4}\sqrt{\sum_{j=1}^{n_2}\ln\BH_\SV^{(j,:)}\rn_2^4}\\
&\leq \ln\BH_\SU\rn_F^2\ln\BH_\SV\rn_F^2+\sqrt{\frac{24n\log(n)}{p}}\lb\ln\BH_\SU\rn_{2,\infty}\ln\BH_\SU\rn_F\rb\lb\ln\BH_\SV\rn_{2,\infty}\ln\BH_\SV\rn_F\rb\\
&\leq \ln\BH_\SU\rn_F^2\ln\BH_\SV\rn_F^2+\sqrt{\frac{24n\log(n)}{p}}\lb\frac{4\mu c_sr}{n}\sigma\rb\ln\BH_\SU\rn_F\ln\BH_\SV\rn_F\\
&\leq \frac{1}{4}\ln\BH\rn_F^4+\sqrt{\frac{96\mu^2 c_s^2r^2\log(n)}{m}}\frac{\sigma_1(\BL^0)}{1-\varepsilon_0}\ln\BH\rn_F^2\\
&\leq \lb\frac{3\varepsilon_0^2}{4}+\frac{\varepsilon_0(1+\varepsilon_0)}{1-\varepsilon_0}\rb\sigma_r(\G\by)\ln\BH\rn_F^2,
\end{align*}
where the fourth line follows from Lemma~\ref{lem:key},  the sixth line follows from 
\begin{align*}
\max\lcb\ln\BH_{\SU}\rn_{2,\infty},\ln\BH_{\SV}\rn_{2,\infty}\rcb=\ln\BH\rn_{2,\infty}\leq\ln\BY\rn_{2,\infty}+\ln\BZ\rn_{2,\infty}\leq 2\sqrt{\frac{\mu c_sr}{n}\sigma},
\end{align*}
and the last line follows from  \eqref{eq:initial_eq1} and the assumptions on $\ln\BH\rn_F^2$ and $m$.

\textbf{Lower bound for $\Real\la \nabla f(\BZ),\BH\ra$.} 
Before finally showing the lower bound for $\Real\la \nabla f(\BZ),\BH\ra$, we need to define the tangent space of the rank $r$ matrix manifold at $\G\by$, denoted $T$. Given the SVD $\G\by=\BU\BS\BV^*$,  we define $T$ as
\begin{align*}
T=\{\BU\BC^*+\BD\BV^*~|~\BC\in\C^{n_2\times r},~\BD\in\C^{n_1\times r}\}.
\end{align*}
One can easily see that $\BH_\SU\BY_\SV^*+\BY_\SU\BH_\SV^*\in T$. 
Substituting the bound for $p^{-1}\sum_{k=1}^m\lab\la \BG_{a_k},\BH_\SU\BH_\SV^*\ra\rab^2$ into \eqref{eq:grad_f_i2} and then combining the lower bounds for $I_1$ and $I_2$ together yields
\begin{align*}
&\Real\la \nabla f(\BZ),\BH\ra\\&\geq \frac{11}{20}\ln (\I-\G\G^*)(\BH_\SU\BY_\SV^*+\BY_\SU\BH_\SV^*)\rn_F^2-3\ln (\I-\G\G^*)(\BH_\SU\BH_\SV^*)\rn_F^2\\
&\quad+\frac{11}{20}\la p^{-1}\P_\Omega\G^*(\BH_\SU\BY_\SV^*+\BY_\SU\BH_\SV^*),\G^*(\BH_\SU\BY_\SV^*+\BY_\SU\BH_\SV^*) \ra\\
&\quad-\lb\frac{9\varepsilon_0^2}{4}+\frac{3\varepsilon_0(1+\varepsilon_0)}{1-\varepsilon_0}\rb\sigma_r(\G\by)\ln\BH\rn_F^2\\
&\geq\frac{11}{20}\ln \BH_\SU\BY_\SV^*+\BY_\SU\BH_\SV^*\rn_F^2- 3\ln \BH_\SU\BH_\SV^*\rn_F^2-\lb\frac{9\varepsilon_0^2}{4}+\frac{3\varepsilon_0(1+\varepsilon_0)}{1-\varepsilon_0}\rb\sigma_r(\G\by)\ln\BH\rn_F^2\\
&\quad -\frac{11}{20}\la \G(\I-p^{-1}\P_{\Omega})\G^*(\BH_\SU\BY_\SV^*+\BY_\SU\BH_\SV^*), \BH_\SU\BY_\SV^*+\BY_\SU\BH_\SV^*\ra\\
&=\frac{11}{20}\ln \BH_\SU\BY_\SV^*+\BY_\SU\BH_\SV^*\rn_F^2-3\ln \BH_\SU\BH_\SV^*\rn_F^2-\lb\frac{9\varepsilon_0^2}{4}+\frac{3\varepsilon_0(1+\varepsilon_0)}{1-\varepsilon_0}\rb\sigma_r(\G\by)\ln\BH\rn_F^2\\
&\quad-\frac{11}{20}\la \P_{T}\G(\I-p^{-1}\P_{\Omega})\G^*\P_T(\BH_\SU\BY_\SV^*+\BY_\SU\BH_\SV^*), \BH_\SU\BY_\SV^*+\BY_\SU\BH_\SV^*\ra\\
&\geq \frac{11}{20}\lb1-\varepsilon_0\rb\ln \BH_\SU\BY_\SV^*+\BY_\SU\BH_\SV^*\rn_F^2-\lb\frac{9\varepsilon_0^2}{2}+\frac{3\varepsilon_0(1+\varepsilon_0)}{1-\varepsilon_0}\rb\sigma_r(\G\by)\ln\BH\rn_F^2\\
&\geq \frac{1}{2}\ln \BH_\SU\BY_\SV^*+\BY_\SU\BH_\SV^*\rn_F^2-\lb\frac{9\varepsilon_0^2}{2}+\frac{3\varepsilon_0(1+\varepsilon_0)}{1-\varepsilon_0}\rb\sigma_r(\G\by)\ln\BH\rn_F^2\\
&\geq \frac{1}{8}\sigma_r(\G\by)\ln\BH\rn_F^2+\Real\la \BH_\SU^*\BY_\SU,\BY_\SV^*\BH_\SV \ra,\numberthis\label{eq:f_lower}
\end{align*}
where the second inequality follows from the fact $\G\G^*$ is a projection operator, the third inequality holds with probability at least $1-n^{-2}$ (see Lemma~\ref{lem:A3}) under the assumption on $m$ and $\ln\BH\rn_F^2$, and the last inequality follows from \kw{$\ln\BH_\SU\BY_\SV^*\rn_F\ge \sigma_r(\BY_{\SV})\ln\BH_\SU\rn_F$,  $\ln\BY_\SU\BH_\SV^*\rn_F\ge \sigma_r(\BY_{\SU})\ln\BH_\SV\rn_F$, and the assumption $\varepsilon_0\leq\frac{1}{11}$}.

\textbf{Lower bound for $\Real\la \nabla F(\BZ),\BH\ra$.} Let $\lambda=\frac{1}{4}$. Combining the  lower bound in \eqref{eq:f_lower} for $\Real\la \nabla f(\BZ),\BH\ra$ and the lower bound  in \eqref{eq:mean_g} for $\Real\la \nabla g(\BZ),\BH\ra$ together gives 
\begin{align*}
\Real\la \nabla F(\BZ),\BH\ra&\geq \frac{1}{8}\sigma_r(\G\by)\ln\BH\rn_F^2-\frac{7}{8}\ln\BH^*\BD\BH\rn_F^2+\frac{1}{8}\ln\BY^*\BD\BH\rn_F^2\\
&\geq \frac{1}{10}\sigma_r(\G\by)\ln\BH\rn_F^2+\frac{1}{8}\ln\BY^*\BD\BH\rn_F^2,
\end{align*}
where the second inequality follows from 
\begin{align*}
\ln\BH^*\BD\BH\rn_F^2\leq \ln\BH\rn_F^4\leq 3\varepsilon_0^2\sigma_r(\G\by)\ln\BH\rn_F^2
\end{align*}
and the assumption $\varepsilon_0\leq\frac{1}{11}$. This concludes the proof of \eqref{eq:local_cur}.
\subsubsection{Proof of \eqref{eq:local_smooth}}
Since \begin{align*}\ln \nabla F(\BZ)\rn_F^2\leq 2\ln \nabla f(\BZ)\rn_F^2+2\lambda^2\ln \nabla g(\BZ)\rn_F^2,\numberthis\label{eq:upper_F}\end{align*} it suffices to bound $\ln \nabla f(\BZ)\rn_F^2$ and $\ln \nabla g(\BZ)\rn_F^2$ separately. 



\textbf{Upper bound for $\ln \nabla g(\BZ)\rn_F^2$.} 
We begin with the upper bound for $\ln \nabla g(\BZ)\rn_F^2$, which can be obtained in a straightforward way,
\begin{align*}
\ln \nabla g(\BZ)\rn_F^2&=\ln\BD\BZ\BZ^*\BD\BZ\rn_F^2=\ln\BD(\BZ\BZ^*-\BY\BY^*)\BD\BZ+\BD\BY\BY^*\BD\BZ\rn_F^2\\
&\leq 2\ln\BD(\BZ\BZ^*-\BY\BY^*)\BD\BZ\rn_F^2+2\ln\BD\BY\BY^*\BD\BZ\rn_F^2\\
&\leq 2\ln\BZ\rn_2^2\ln\BZ\BZ^*-\BY\BY^*\rn_F^2+2\ln\BY\rn_2^2\ln\BY^*\BD(\BY+\BH)\rn_F^2\\
&=2\ln\BZ\rn_2^2\ln\BY\BH^*+\BH\BY^*+\BH\BH^*\rn_F^2+2\ln\BY\rn_2^2\ln\BY^*\BD\BH\rn_F^2\\
&\leq 6\ln\BZ\rn_2^2\lb2\ln\BY\rn_2^2\ln\BH\rn_F^2+\ln\BH\rn_F^4\rb+2\ln\BY\rn_2^2\ln\BY^*\BD\BH\rn_F^2\\
&\leq \kw{6\lb\sqrt{3\varepsilon_0^2\sigma_r(\G\by)}+\sqrt{2\sigma_1(\G\by)}\rb^2\lb4\sigma_1(\G\by)+3\varepsilon_0^2\sigma_r(\G\by)\rb\ln\BH\rn_F^2}\\
&\quad \kw{+4\sigma_1(\G\by)\ln\BY^*\BD\BH\rn_F^2}\\
&\leq 60\sigma_1^2(\G\by)\ln\BH\rn_F^2+4\sigma_1(\G\by)\ln\BY^*\BD\BH\rn_F^2,\numberthis\label{eq:upper_g}
\end{align*}
where the third equality follows from $\BZ=\BY+\BH$ and $\BY^*\BD\BY=\bzero$, the fourth inequality follows from $\ln\BY\rn_2=\sqrt{2\sigma_1(\G\by)}$ and
\begin{align*}
\ln\BZ\rn_2\leq \ln\BZ-\BY\rn_2+\ln\BY\rn_2\leq \ln\BZ-\BY\rn_F+\ln\BY\rn_2,
\end{align*}
and the last line follows from the assumption $\varepsilon_0\leq\frac{1}{11}$.

In order to bound $\ln\nabla f(\BZ)\rn_F^2$, we consider $\lab\la \nabla f(\BZ),\BX\ra\rab^2$ for matrices $\BX=\begin{bmatrix}\BX_\SU&\BX_\SV\end{bmatrix}^T$ with unit Frobenius norm (i.e., $\ln\BX_\SU\rn_F^2+\ln\BX_\SV\rn_F^2=1$).
Note that 
\begin{align*}
\lab\la \nabla f(\BZ),\BX\ra\rab^2&=\lab\la (\I-\G\G^*)(\BZ_\SU\BZ_\SV^*)+p^{-1}\G\P_\Omega\G^*(\BZ_\SU\BZ_\SV^*-\BY_\SU\BY_\SV^*),\BX_\SU\BZ_\SV^*+\BZ_\SU\BX_\SV^*\ra\rab^2\\
&=\lab\la (\I-\G\G^*)(\BZ_\SU\BZ_\SV^*-\BY_\SU\BY_\SV^*)+p^{-1}\G\P_\Omega\G^*(\BZ_\SU\BZ_\SV^*-\BY_\SU\BY_\SV^*),\BX_\SU\BZ_\SV^*+\BZ_\SU\BX_\SV^*\ra\rab^2\\
&\leq2\lab \la(\I-\G\G^*)(\BZ_\SU\BZ_\SV^*-\BY_\SU\BY_\SV^*), \BX_\SU\BZ_\SV^*+\BZ_\SU\BX_\SV^*\ra\rab^2 \\
&\quad+2\lab \la p^{-1}\G\P_\Omega\G^*(\BZ_\SU\BZ_\SV^*-\BY_\SU\BY_\SV^*),\BX_\SU\BZ_\SV^*+\BZ_\SU\BX_\SV^*\ra\rab^2\\
&=2\lab \la(\I-\G\G^*)(\BZ_\SU\BH_\SV^*+\BH_\SU\BY_\SV^*), \BX_\SU\BZ_\SV^*+\BZ_\SU\BX_\SV^*\ra\rab^2 \\
&\quad+2\lab \la p^{-1}\G\P_\Omega\G^*(\BZ_\SU\BH_\SV^*+\BH_\SU\BY_\SV^*),\BX_\SU\BZ_\SV^*+\BZ_\SU\BX_\SV^*\ra\rab^2\\
&:=2\cdot I_3+2\cdot I_4.\numberthis\label{eq:upper_f}
\end{align*}

\textbf{Upper bound for $I_3$. } 
Since
\begin{align*}
&\ln\BZ_\SU\rn_2\leq \ln\BY_\SU\rn_2+\ln\BH_\SU\rn_2\leq \ln\BY_\SU\rn_2+\ln\BH\rn_F\leq (1+\sqrt{3}\varepsilon_0)\sqrt{\sigma_1(\G\by)},\\
&\ln\BZ_\SV\rn_2\leq \ln\BY_\SV\rn_2+\ln\BH_\SV\rn_2\leq \ln\BY_\SU\rn_2+\ln\BH\rn_F\leq (1+\sqrt{3}\varepsilon_0)\sqrt{\sigma_1(\G\by)}.
\end{align*}
one has
\begin{align*}
\ln\BZ_\SU\BH_\SV^*+\BH_\SU\BY_\SV^*\rn_F^2&\leq 2\lb \ln \BZ_\SU\BH_\SV^*\rn_F^2+\ln\BH_\SU\BY_\SV^*\rn_F^2\rb\\
&\leq 2\lb\ln\BZ_\SU\rn_2^2\ln\BH_\SV\rn_F^2+\ln\BY_\SV\rn_2^2\ln\BH_\SU\rn_F^2\rb\\
&\leq 2 (1+\sqrt{3}\varepsilon_0)^2\sigma_1(\G\by)\ln\BH\rn_F^2
\end{align*}
and 
\begin{align*}
\ln\BX_\SU\BZ_\SV^*+\BZ_\SU\BX_\SV^*\rn_F^2&\leq 2\lb\ln \BX_\SU\BZ_\SV^*\rn_F^2+\ln \BZ_\SU\BX_\SV^*\rn_F^2\rb\\
&\leq 2\lb \ln\BZ_\SV\rn_2^2\ln\BX_\SU\rn_F^2+\ln\BZ_\SU\rn_2^2\ln\BX_\SV\rn_F^2\rb\\
&\leq 2 (1+\sqrt{3}\varepsilon_0)^2\sigma_1(\G\by),
\end{align*}
where in the last line we have utilized $\ln\BX_\SU\rn_F^2+\ln\BX_\SV\rn_F^2=1$. Because $\I-\G\G^*$ is a projection operator, $I_3$ can be bounded as follows:
\begin{align*}
I_3&\leq \ln\BZ_\SU\BH_\SV^*+\BH_\SU\BY_\SV^*\rn_F^2\cdot\ln\BX_\SU\BZ_\SV^*+\BZ_\SU\BX_\SV^*\rn_F^2\leq 
4(1+\sqrt{3}\varepsilon_0)^4\sigma_1^2(\G\by)\ln\BH\rn_F^2.
\end{align*}

\textbf{Upper bound for $I_4$.} Notice that 
\begin{align*}
&\lab \la p^{-1}\G\P_\Omega\G^*(\BZ_\SU\BH_\SV^*+\BH_\SU\BY_\SV^*),\BX_\SU\BZ_\SV^*+\BZ_\SU\BX_\SV^*\ra\rab\\
&\leq \lab \la p^{-1}\G\P_\Omega\G^*(\BZ_\SU\BH_\SV^*),\BX_\SU\BZ_\SV^*\ra\rab+\lab \la p^{-1}\G\P_\Omega\G^*(\BZ_\SU\BH_\SV^*),\BZ_\SU\BX_\SV^*\ra\rab\\
&\quad+ \lab \la p^{-1}\G\P_\Omega\G^*(\BH_\SU\BY_\SV^*),\BX_\SU\BZ_\SV^*\ra\rab+\lab \la p^{-1}\G\P_\Omega\G^*(\BH_\SU\BY_\SV^*),\BZ_\SU\BX_\SV^*\ra\rab.\numberthis\label{eq:upper_i4}
\end{align*}
We can bound $\lab \la p^{-1}\G\P_\Omega\G^*(\BZ_\SU\BH_\SV^*),\BX_\SU\BZ_\SV^*\ra\rab$ as follows: 
\begin{align*}
&\lab \la p^{-1}\G\P_\Omega\G^*(\BZ_\SU\BH_\SV^*),\BX_\SU\BZ_\SV^*\ra\rab\\
&=p^{-1}\lab \la \P_\Omega\G^*(\BZ_\SU\BH_\SV^*),\G^*(\BX_\SU\BZ_\SV^*)\ra\rab\\
&\leq p^{-1}\sum_{k=1}^m\lcb\lab\la \BG_{a_k},\BZ_\SU\BH_\SV^*\ra\rab\lab\la\BG_{a_k},\BX_\SU\BZ_\SV^*\ra\rab\rcb\\
&= p^{-1}\sum_{k=1}^m\lcb\lab\frac{1}{\sqrt{w_{a_k}}}\sum_{i+j=a_k}\la\be_i\be_j^T,\BZ_\SU\BH_\SV^*\ra\rab\lab\frac{1}{\sqrt{w_{a_k}}}\sum_{i+j=a_k}\la\be_i\be_j^T,\BX_\SU\BZ_\SV^*\ra\rab\rcb\\
&\leq p^{-1}\sum_{k=1}^m\lcb\frac{1}{\sqrt{w_{a_k}}}\sum_{i+j=a_k}\lab\la\be_i\be_j^T,\BZ_\SU\BH_\SV^*\ra\rab\frac{1}{\sqrt{w_{a_k}}}\sum_{i+j=a_k}\lab\la\be_i\be_j^T,\BX_\SU\BZ_\SV^*\ra\rab\rcb\\
&=p^{-1}\sum_{k=1}^m\lcb\sqrt{\lb\frac{1}{\sqrt{w_{a_k}}}\sum_{i+j=a_k}\lab\la\be_i\be_j^T,\BZ_\SU\BH_\SV^*\ra\rab\rb^2}\sqrt{\lb\frac{1}{\sqrt{w_{a_k}}}\sum_{i+j=a_k}\lab\la\be_i\be_j^T,\BX_\SU\BZ_\SV^*\ra\rab\rb^2}\rcb\\
&\leq p^{-1}\sum_{k=1}^m\lcb \lb \sum_{i+j=a_k}\lab\la\be_i\be_j^T,\BZ_\SU\BH_\SV^*\ra\rab^2\rb^{1/2}\lb \sum_{i+j=a_k}\lab\la\be_i\be_j^T,\BX_\SU\BZ_\SV^*\ra\rab^2\rb^{1/2}\rcb\\
&\leq p^{-1}\sum_{k=1}^m\lcb \lb \sum_{i+j=a_k}\ln\BZ_{\SU}^{(i,:)}\rn_2^2\ln\BH_{\SV}^{(j,:)}\rn_2^2\rb^{1/2}\lb \sum_{i+j=a_k}\ln\BX_{\SU}^{(i,:)}\rn_2^2\ln\BZ_{\SV}^{(j,:)}\rn_2^2\rb^{1/2}\rcb\\
&\leq p^{-1}\sum_{k=1}^m\lcb \lb\ln\BZ\rn_{2,\infty}\ln\BH_\SV\rn_F\rb\lb \ln\BZ\rn_{2,\infty}\ln\BX_{\SU}\rn_F\rb\rcb\\
&\leq \mu c_sr\sigma\ln\BH_\SV\rn_F\ln\BX_{\SU}\rn_F,
\end{align*}
where in the last line, we utilize $\ln\BZ\rn_{2,\infty}^2\leq \mu c_sr\sigma/n$. Similar upper bounds can be established for the other three terms in \eqref{eq:upper_i4}. That is,
\begin{align*}
&\lab \la p^{-1}\G\P_\Omega\G^*(\BZ_\SU\BH_\SV^*),\BZ_\SU\BX_\SV^*\ra\rab\leq \mu c_s r\sigma  \ln\BH_\SV\rn_F\ln\BX_\SV\rn_F,\\
&\lab \la p^{-1}\G\P_\Omega\G^*(\BH_\SU\BY_\SV^*),\BX_\SU\BZ_\SV^*\ra\rab\leq \mu c_sr\sigma \ln\BH_\SU\rn_F\ln\BX_\SU\rn_F,\\
&\lab \la p^{-1}\G\P_\Omega\G^*(\BH_\SU\BY_\SV^*),\BZ_\SU\BX_\SV^*\ra\rab\leq \mu c_sr\sigma \ln\BH_\SU\rn_F\ln\BX_\SV\rn_F.
\end{align*}
Combining these four upper bounds together yields
\begin{align*}
I_4&\leq (\mu c_sr\sigma)^2 \lb \ln\BH_\SV\rn_F\ln\BX_{\SU}\rn_F+\ln\BH_\SV\rn_F\ln\BX_\SV\rn_F+\ln\BH_\SU\rn_F\ln\BX_\SU\rn_F+\ln\BH_\SU\rn_F\ln\BX_\SV\rn_F\rb^2\\
&=(\mu c_sr\sigma)^2\lb \ln\BH_\SU\rn_F+\ln\BH_\SV\rn_F\rb^2\lb \ln\BX_\SU\rn_F+\ln\BX_\SV\rn_F\rb^2\\
&\leq 4(\mu c_s r\sigma)^2\ln\BH\rn_F^2
\end{align*}
where in the last line we have used  the fact $\ln\BX_\SU\rn_F^2+\ln\BX_\SV\rn_F^2=1$.

\textbf{Upper bound for $\ln\nabla f(\BZ)\rn_F^2$}. Substituting the upper bounds for $I_3$ and $I_4$ into \eqref{eq:upper_f} give the upper bound for $\ln\nabla f(\BZ)\rn_F^2$,
\begin{align*}
\ln\nabla f(\BZ)\rn_F^2\leq 8\lb(1+\sqrt{3}\varepsilon_0)^4\sigma_1^2(\G\by)+(\mu c_sr\sigma)^2\rb\ln\BH\rn_F^2.\end{align*}

\textbf{Upper bound for $\ln\nabla F(\BZ)\rn_F^2$}. Noting $\lambda=1/4$, $\sigma\leq (1+\varepsilon_0)\sigma_1(\G\by)/(1-\varepsilon_0)$, and $\varepsilon_0\leq 1/11$, after substituting the upper bounds for $\ln\nabla f(\BZ)\rn_F^2$ and $\ln\nabla g(\BZ)\rn_F^2$ into \eqref{eq:upper_F}, we get
\begin{align*}
\ln\nabla F(\BZ)\rn_F^2&\leq 16\lb(1+\sqrt{3}\varepsilon_0)^4\sigma_1^2(\G\by)+(\mu c_sr\sigma)^2+\frac{60}{128}\sigma_1^2(\G\by)\rb\ln\BH\rn_F^2+\frac{1}{2}\sigma_1(\G\by)\ln\BY^*\BD\BH\rn_F^2\\
&\leq 60(\mu c_s r)^2 \sigma_1^2(\G\by)\ln\BH\rn_F^2+\frac{1}{2}\sigma_1(\G\by)\ln\BY^*\BD\BH\rn_F^2,
\end{align*}
which completes the proof of \eqref{eq:local_smooth}.

%% file: discussion.tex
\section{Discussion}\label{sec:discuss}
We have proposed a novel algorithm for spectral compressed sensing by applying projected gradient descent updates to a non-convex functional. Exact recovery guarantee has been established, showing that $O(r^2\log(n))$ random observations are sufficient for the algorithm to achieve the successful recovery.  Additionally, empirical evaluation shows that our algorithm is competitive with other state-of-the-art algorithms. In particular, our algorithm is superior to FIHT, a non-convex algorithm for spectral compressed sensing with provable  recovery guarantees, in terms of phase transitions when the number of observations is small.

For future work,  recovery stability of the proposed algorithm to additive noise will be investigated.  The proofs presented in this paper should extend easily to bounded noise with a small magnitude. It remains to address whether  or not our  algorithm can achieve some statistically optimal rates   under a stochastic noise model.

Recently, a line of research work has been devoted to the geometric analysis of non-convex optimization problems including dictionary learning \cite{juge}, phase retrieval \cite{juge02}, low rank matrix sensing and matrix completion \cite{BNS2016Globa,park_saddle,ge_matrix,ge_matrix2}, tensor completion \cite{ge_tensor} and robust PCA \cite{ge_matrix2}. It has been shown that the non-convex functionals for those problems have well-behaved landscape: all local minima  are also globally optimal. Preliminary numerical results show that our projected gradient descent algorithm works equally well with random initialization, which suggests  the geometric landscape of the objective function $F(\BZ)$ introduced in this paper may  be similarly well-behaved. 

%% file: appendix.tex
\appendix
\section{Supplementary  Lemmas}
Here we list three technical lemmas from the literature that have been used in the analysis of PGD.
\begin{lemma}[{\cite{FIHT}, Lemma 2}]\label{lem:A1}
Assume $\G{\by}$ is $\mu_0$-incoherent and let $\BL_0=\T_r\G(p^{-1}\P_{\Omega}{(\by)})$. Then,
$$
\|\BL_0-\G\by\|_2\lesssim\sqrt{\frac{\mu_0c_sr\log(n)}{m}}\|\G{\by}\|_2
$$
holds with probability at least $1-n^{-2}$.\end{lemma}

\begin{lemma}[{\cite[Lemma~5.4]{PF}}]\label{lem:A2}
For any $\BZ,~\BX\in\C^{(n+1)\times r}$, one has 
\begin{align*}
\distsq{\BZ}{\BX}\leq \frac{1}{2(\sqrt{2}-1)\sigma_r^2(\BX)}\ln\BZ\BZ^*-\BX\BX^*\rn_F^2.
\end{align*} 
\end{lemma}

\begin{lemma}[{\cite{Chi}, Lemma 3}]\label{lem:A3}
Assume $\G{\by}$ is $\mu_0$-incoherent, and let $T$ be the tangent space of the rank $r$ matrix manifold 
at $\G\by$. Then,
\begin{equation*}
\|\P_{T}\G(\I-p^{-1}\P_{\Omega})\G^*\P_{T}\|_2\leq\sqrt{\frac{32\mu_0c_sr\log(n)}{m}}
\end{equation*}
holds with probability at least $1-n^{-2}$.
\end{lemma}